\documentclass[11pt]{amsart}

\usepackage{amsmath}
\usepackage{amsthm}
\usepackage{amssymb}
\usepackage{tikz}
\usepackage{mathtools}
\usepackage{multirow}

\allowdisplaybreaks

\numberwithin{equation}{section}

\newcommand{\Z}{\mathbb{Z}}
\newcommand{\N}{\mathbb{N}}

\newcommand{\R}{\mathbb{R}}
\newcommand{\F}{\mathcal{F}}
\newcommand{\Sh}{\mathcal{S}}

\newcommand{\Nu}{\boldsymbol{\nu}}
\newcommand{\K}{\boldsymbol{k}}
\newcommand{\A}{\mathcal{A}}

\newcommand{\supp}{\mathop{\mathrm{supp}}}
\newcommand{\pa}{\partial}
\newcommand{\vphi}{\varphi}
\newcommand{\Op}{\mathop{\mathrm{Op}}}

\newcommand{\la}{\langle}
\newcommand{\ra}{\rangle}

\theoremstyle{plain}
\newtheorem{thm}{Theorem}[section]
\newtheorem{prop}[thm]{Proposition}
\newtheorem{lem}[thm]{Lemma}

\newtheorem*{thmA}{Theorem A}
\newtheorem*{thmB}{Theorem B}

\theoremstyle{definition}
\newtheorem{dfn}[thm]{Definition}
\newtheorem{rem}[thm]{Remark}

\begin{document}

\title[Bilinear pseudo-differential operators]
{Boundedness of bilinear pseudo-differential operators with $BS^{m}_{0,0}$ symbols 
on Sobolev spaces}

\author[N. Shida]{Naoto Shida}

\date{\today}

\address[N. Shida]
{Graduate School of Mathematics, Nagoya University, 
Chikusa-ku, Nagoya 464-8602, Japan}

\email[N. Shida]{naoto.shida.c3@math.nagoya-u.ac.jp}



%=============================================================================
%=============================================================================
%=============================================================================
%=============================================================================
%=============================================================================
%=============================================================================
%=============================================================================
%=============================================================================
%=============================================================================
\begin{abstract}
In the present paper, 
bilinear pseudo-differential operators with symbols 
in the bilinear H\"ormander class $BS^m_{0,0}$ are considered.
In particular, the boundedness of these operators on Sobolev spaces is established. 
Our main result is proved by using symbolic calculus and 
the boundedness  of those operators 
with certain $S_{0,0}$-type symbols on Lebesgue spaces.
\end{abstract}
%=============================================================================
%=============================================================================
%=============================================================================
%=============================================================================
%=============================================================================
%=============================================================================
%=============================================================================
%=============================================================================
%=============================================================================
\maketitle
%=============================================================================
%=============================================================================
%=============================================================================
%=============================================================================
%=============================================================================
%=============================================================================
%=============================================================================
%=============================================================================
%=============================================================================
%
\section{Introduction}\label{Intro}
For a  bounded function $\sigma = \sigma(x, \xi_1, \xi_2)$ on $(\R^n)^3$, 
the bilinear pseudo-differential operator $T_{\sigma}$ is defined by
\[
T_{\sigma}(f_1, f_2)(x)
=
\frac{1}{(2\pi)^{2n}}
\int_{(\R^n)^2}
e^{ix \cdot (\xi_1 + \xi_2)}
\sigma(x, \xi_1, \xi_2)
\widehat{f}_1(\xi_1)
\widehat{f}_2(\xi_2)
\,
d\xi_1
d\xi_2,
\quad
x \in \R^n,
\]
where $f_1, f_2 \in \Sh(\R^n)$.
If $\sigma$ is independent of the variable $x$, that is, $\sigma = \sigma(\xi_1, \xi_2)$, then we call $T_{\sigma}$ a bilinear Fourier multiplier operator.

%=============================================================================
%=============================================================================
%=============================================================================
Let $X_1$, $X_2$ and $Y$ be function spaces on $\R^n$ equipped with quasi-norms $\|\cdot\|_{X_1}$, $\|\cdot\|_{X_2}$ and $\|\cdot\|_{Y}$, respectively.
If there exists a positive constant $C$ such that
\begin{equation}\label{bdd-dfn}
\|T_{\sigma}(f_1, f_2)\|_{Y} \le C \|f_1\|_{X_1} \|f_2\|_{X_2}
\end{equation}
for all $f_j \in \Sh \cap X_j$, $j=1, 2$, then we say that $T_{\sigma}$ is bounded from $X_1 \times X_2$ to $Y$.
We write the smallest constant $C$ of \eqref{bdd-dfn} as $\|T_{\sigma}\|_{X_1 \times X_2 \to Y}$.
For a  symbol class $\A$, $\Op(\A)$ denotes the class of all operators $T_{\sigma}$ with $\sigma \in \A$.  
If $T_{\sigma}$ is bounded from $X_1 \times X_2$ to $Y$ for all $\sigma \in \A$, then
we write $\Op(\A) \subset B(X_1 \times X_2 \to Y)$.
%=============================================================================

%=============================================================================
In the present paper, we consider the following symbol class.
\begin{dfn}
For $m \in \R$, the class $BS^{m}_{0,0}$ denotes the set of all 
$\sigma = \sigma(x, \xi_1, \xi_2) \in C^{\infty}((\R^n)^3)$ such that
\[
|
\pa_{x}^{\alpha}
\pa_{\xi_1}^{\beta_1}
\pa_{\xi_2}^{\beta_2}
\sigma(x, \xi_1, \xi_2)
|
\le
C_{\alpha, \beta_1, \beta_2}
(1+|\xi_1|+|\xi_2|)^{m}
\] 
for all multi-indices $\alpha, \beta_1, \beta_2 \in \N_0^n = \{0, 1, 2, \dots\}^n$.
\end{dfn}
%\noindent
%We remark that the class $BS^m_{0,0}$ is a particular case of the bilinear H\"ormander classes.
%=============================================================================
It was first proved by B\'enyi-Torres \cite{BT-2} that 
for $1 \le p_1, p_2, p < \infty$, $1/p=1/p_1+1/p_2$, 
there exists a symbol in $BS^0_{0,0}$ such that the corresponding bilinear Fourier multiplier operator is not bounded from $L^{p_1} \times L^{p_2}$ to $L^p$.
In particular, the boundedness 
\[
\Op(BS^{0}_{0,0}) \subset B(L^2 \times L^2 \to L^1)
\]
does not hold in contrast to the celebrated Calder\'on-Vaillancourt theorem \cite{CV} for linear pseudo-differential operators.
After the works of B\'enyi-Bernicot-Maldonado-Naibo-Torres \cite{BBMNT}, Michalowski-Rule-Staubach \cite{MRS} and Miyachi-Tomita \cite{MT-IUMJ},  the following result was given by Kato-Miyachi-Tomita \cite{KMT-multi-2022} very recently.

\begin{thmA}[{\cite[Theorem 1.2]{KMT-multi-2022}}]
Let  $0< p_1, p_2, p \le \infty$, $1/p \le 1/p_1+1/p_2$, and $m \in \R$.
Then the boundedness
\begin{equation}\label{boundedness-hp}
\Op(BS^{m}_{0,0})
\subset
B(h^{p_1} \times h^{p_2} \to h^p)
\end{equation}
holds if and only if 
\begin{equation}\label{criticalorder}
m
\le
\min
\left\{
\frac{n}{p},
\frac{n}{2}
\right\}
-
\max
\left\{
\frac{n}{p_1},
\frac{n}{2}
\right\}
-
\max
\left\{
\frac{n}{p_2},
\frac{n}{2}
\right\}.
\end{equation}
If \eqref{criticalorder} is satisfied and if $p_1 = \infty$ (resp. $p_2=\infty$), then \eqref{boundedness-hp} holds with $h^{p_1}$ (resp. $h^{p_2}$)  
replaced by $bmo$. 
\end{thmA}
\noindent
For the definition of the spaces $h^p$ and $bmo$, see Section \ref{prelim}.

%=============================================================================
%=============================================================================
%=============================================================================
Next, we recall the boundedness results of linear pseudo-differential operators. 
For a bounded function $\sigma = \sigma(x, \xi)$ on $(\R^n)^2$, 
we define the linear pseudo-differential operator
by
\[
\sigma(X, D)f(x)
=
\frac{1}{(2\pi)^n}
\int_{\R^n}
e^{ix \cdot \xi}
\sigma(x, \xi)
\widehat{f}(\xi)
\,
d\xi,
\quad
x \in \R^n.
\]
The linear H\"ormander symbol class $S^{m}_{0, 0}$ is defined by the set of all 
$\sigma = \sigma(x, \xi) \in C^{\infty}((\R^n)^2)$ such that
\[
|\pa_{x}^{\alpha} \pa_{\xi}^{\beta} \sigma(x, \xi)|
\le
C_{\alpha, \beta}
(1+|\xi|)^{m}
\]
for all $\alpha, \beta \in \N^n_0$.

The boundedness  of  linear pseudo-differential operators with symbols in $S^m_{0,0}$
are well studied by Calder\'on-Vaillancourt \cite{CV}, 
Coifman-Meyer \cite{CM},  
Miyachi \cite{Miyachi}
and
P\"aiv\"arinta-Somersalo \cite{PS}, 
for instance.
More precisely, it is known that the following theorem holds.

\begin{thmA}[\cite{CV, CM, Miyachi, PS}]
Let $0< p \le \widetilde{p} \le \infty$ and $m \in \R$. Then the boundedness
\[
\Op(S^{m}_{0,0})
\subset
B(h^{p} \to h^{\widetilde{p}})
\]
holds if and only if
\[
m 
\le 
\min
\left\{
\frac{n}{\widetilde{p}},
\frac{n}{2}
\right\}
-
\max
\left\{
\frac{n}{p},
\frac{n}{2}
\right\},
\]
where $h^{p}$ (resp. $h^{\widetilde{p}}$) should be replaced by $bmo$ 
if $p = \infty$ (resp. $\widetilde{p} = \infty$).
\end{thmA}

\noindent
We remark that the proof of  the `only if' part can be found in \cite[Theorem 1.5]{KMT-multi-2022}.

For the linear case, the boundedness of $\Op(S^m_{0,0})$ on  Sobolev  spaces
can be derived from Theorem A and symbolic calculus in $S^m_{0, 0}$.
It is known that 
if $\sigma_j \in S^{m_j}_{0,0}$, $j=1,2$, then there exist $\sigma_0 \in S^{m_1+m_2}_{0,0}$ such that 
\[
\sigma_1(X, D)\sigma_2(X, D) = \sigma_0(X, D)
\]
(see, e.g., \cite[Chapter VI\hspace{-1pt}I \S 5]{Stein-Harmonic}).
Thus, we see that 
\[
(I-\Delta)^{\widetilde{s}/2} \sigma(X, D) (I-\Delta)^{-s/2}
\in
\Op(S^{m-s+\widetilde{s}}_{0, 0})
\]
since $(I-\Delta)^{a/2} \in \Op(S^{a}_{0, 0})$ for $a \in \R$. 
Combining this argument with Theorem A, 
we obtain the following.
\begin{thmB}\label{linear-Sobolev}
Let $0< p \le \widetilde{p} \le \infty$, $s, \widetilde{s} \in \R$ and $m \in \R$.
Then the boundedness
\begin{equation}\label{bdd-linear-Sobolev}
\Op(S^m_{0,0})
\subset
B(h^{p}_{s} \to h^{\widetilde{p}}_{\widetilde{s}})
\end{equation}
holds if and only if
\begin{equation}\label{m-linear}
m 
\le 
\min
\left\{
\frac{n}{\widetilde{p}},
\frac{n}{2}
\right\}
-
\max
\left\{
\frac{n}{p},
\frac{n}{2}
\right\}
+s-\widetilde{s}.
\end{equation}
Here $h^{p}_{s}$ (resp. $h^{\widetilde{p}}_{\widetilde{s}}$) should be replaced by $bmo_s$ (resp. $bmo_{\widetilde{s}}$) if $p = \infty$ (resp. $\widetilde{p} = \infty$).
\end{thmB}
%=============================================================================
%=============================================================================
%=============================================================================
The purpose of this paper is to generalize Theorem B to the bilinear case.
Our main result reads as follows.

\begin{thm}\label{main-thm}
Let $0< p_1, p_2, p \le \infty$, $1/p \le 1/p_1+1/p_2$, $s_1, s_2, s \in \R$,
and let
\begin{equation}\label{m-critical}
m 
= 
\min
\left\{
\frac{n}{p},
\frac{n}{2}
\right\}
-
\max
\left\{
\frac{n}{p_1},
\frac{n}{2}
\right\}
-
\max
\left\{
\frac{n}{p_2},
\frac{n}{2}
\right\}
+s_1
+s_2
-s.
\end{equation}
If
\begin{align}
&
s_1
<
\max
\left\{
\frac{n}{p_1},
\frac{n}{2}
\right\},
\quad
s_2
<
\max
\left\{
\frac{n}{p_2},
\frac{n}{2}
\right\},
\quad
s
>
-
\max
\left\{
\frac{n}{p^{\prime}},
\frac{n}{2}
\right\},
\label{m-critical?}
\end{align}
then the boundedness
\begin{equation}\label{main-est}
\Op(BS^{m}_{0,0})
\subset
B(h^{p_1}_{s_1} \times h^{p_2}_{s_2} \to h^p_s)
\end{equation}
holds, where $h^{p_1}_{s_1}$ (resp. $h^{p_2}_{s_2}$, $h^p_s$) should be replaced by 
$bmo_{s_1}$ (resp. $bmo_{s_2}$, $bmo_s$)
if $p_1=\infty$ (resp. $p_2 = \infty$, $p= \infty$).
\end{thm}

Our assumptions \eqref{m-critical} and \eqref{m-critical?} 
are sharp in the following sense.

\begin{thm}\label{main-thm2}
Let $m \in \R$, $0< p_1, p_2, p \le \infty$, $1/p \le 1/p_1+1/p_2$ 
and $s_1, s_2, s \in \R$.
If the boundedness \eqref{main-est} holds, then
\begin{equation}\label{necessity-1}
m
\le
\min
\left\{
\frac{n}{p},
\frac{n}{2}
\right\}
-
\max
\left\{
\frac{n}{p_1},
\frac{n}{2}
\right\}
-
\max
\left\{
\frac{n}{p_2},
\frac{n}{2}
\right\}
+s_1
+s_2
-s
\end{equation}
and
\begin{align}
&s_1 
\le 
\max
\left\{
\frac{n}{p_1},
\frac{n}{2}
\right\}
+\kappa,
\quad
s_2 
\le 
\max
\left\{
\frac{n}{p_2},
\frac{n}{2}
\right\}
+\kappa,
\quad
s
\ge
-\max
\left\{
\frac{n}{p^\prime},
\frac{n}{2}
\right\}
-\kappa
\label{necessity-3},
\end{align}
where $\kappa = \min\{n/p, n/2\}-\max\{n/p_1, n/2\}-\max\{n/p_2, n/2\}+s_1+s_2-s-m$.
\end{thm}

We shall give some remarks on differences between the linear and bilinear cases. 

First, in the linear case, 
the boundedness \eqref{bdd-linear-Sobolev} holds for arbitrary $s$ and $\widetilde{s}$ if $m$ satisfies \eqref{m-linear}. 
However, in the bilinear case, 
we need certain restrictions as in \eqref{m-critical?}, and Theorem \ref{main-thm2} means that they are necessary.
If the order $m$ does not satisfy \eqref{m-critical}, that is, 
\[
m <
\min \{n/p, n/2\} - \max\{n/p_1, n/2\} - \max\{n/p_2, n/2\} 
+s_1
+s_2
-s,
\]
then we can relax those restrictions. 
See Theorem \ref{main-thm+} (2).

Secondly, 
it becomes more difficult than the linear case to treat the boundedness in Sobolev spaces 
for the bilinear case.
Roughly speaking, Theorem B can be derived from Theorem A and symbolic calculus in $S^m_{0,0}$.
However, the same argument 
does not work for the bilinear case.
Indeed, 
by the same reasons as in the linear case,
the $h^{p_1}_{s_1} \times h^{p_2}_{s_2} \to h^{p}_{s}$
boundedness of $T_{\sigma}$ with $\sigma \in BS^m_{0,0}$ 
follows from the $h^{p_1} \times h^{p_2} \to h^{p}$ 
boundedness of $T_{\tau}$
given by
\begin{equation}\label{symbol***}
T_{\tau}(f_1, f_2)
=
(I-\Delta)^{s/2}
T_{\sigma}((I-\Delta)^{-s_1/2}f_1, (I-\Delta)^{-s_2/2}f_2).
\end{equation}
However, 
in contrast to the linear case,
the symbol $\tau$ is not always 
within the framework of
the bilinear H\"ormander symbol class.
For instance,  the function
\[
\sigma(\xi_1, \xi_2) = \sum_{\nu_1, \nu_2 \in \Z^n} (1+|\nu_1|+|\nu_2|)^{m} \vphi(\xi_1 -\nu_1) \vphi(\xi_2 -\nu_2),
\]
where $\vphi$ is a function in $\Sh(\R^n)$ such that 
$\supp \vphi \subset [-1/10, 1/10]^n$, 
belongs to the the class $BS^m_{0,0}$. 
Then $\tau$ given by \eqref{symbol***} can be written as
\[
\tau(\xi_1, \xi_2)
=
(1+|\xi_1+\xi_2|^2)^{s/2}
(1+|\xi_1|^2)^{-s_1/2}
(1+|\xi_2|^2)^{-s_2/2}
\sigma(\xi_1, \xi_2).
\]
Thus we see that this $\tau$ satisfies
\begin{equation}\label{decaynew*}
|\pa_{x}^{\alpha} \pa_{\xi_1}^{\beta_1} \pa_{\xi_2}^{\beta_2} \tau(\xi_1, \xi_2)|
\lesssim
(1+|\xi_1+\xi_2|)^{s}
(1+|\xi_1|)^{-s_1}
(1+|\xi_2|)^{-s_2}
(1+|\xi_1| + |\xi_2|)^m.
\end{equation}
It should be emphasized that this estimate 
cannot be improved.
Hence we need to consider more general symbol classses 
than the bilinear H\"ormander symbol classes $BS^m_{0,0}$.

In order to treat bilinear symbols satisfying \eqref{decaynew*},
we will use some $S_{0,0}$-type 
symbol classes and prove the boundedness of
bilinear pseudo-differential operators with their symbols
(see Section \ref{secS00type} for details).

%===========================================================================

We end this section by explaining the contents of this paper. 
In Section \ref{prelim}, 
we will give the basic notations, the definitions and properties of 
function spaces and a key estimate which will be useful in the proof
of Theorem \ref{main-thm}.
In Section \ref{secS00type}, 
the boundedness results for bilinear pseudo-differential operators
of $S_{0,0}$-type are given.
In Section \ref{sec-mainthm}, 
we introduce more general statement than Theorem \ref{main-thm} and prove it.
In Section \ref{sec-necessity}, 
we give the proof of Theorem \ref{main-thm2}.
In Appendix, 
we explain details on the interpolation argument used in the proof of Theorem \ref{main-thm+}.
%===========================================================================
%===========================================================================
%===========================================================================
\section{Preliminaries}
\label{prelim}
\subsection{Basic notations}
Let $\Sh(\R^n)$ and $\Sh'(\R^n)$ be the Schwartz space of
rapidly decreasing smooth functions on $\R^n$ and 
the space of tempered distributions, respectively.
We define the Fourier transform $\F f$
and the inverse Fourier transform $\F^{-1}f$
of $f \in \Sh(\R^n)$ by
\[
\F f(\xi)
=\widehat{f}(\xi)
=\int_{\R^n}e^{-ix\cdot\xi} f(x)\, dx
\quad \text{and} \quad
\F^{-1}f(x)
=\frac{1}{(2\pi)^n}
\int_{\R^n}e^{i \xi \cdot x} f(\xi)\, d\xi.
\]
For $m \in L^{\infty}(\R^n)$,
the Fourier multiplier operator $m(D)$ is defined by
$m(D)f=\F^{-1}[m\widehat{f}]$ for $f \in \Sh(\R^n)$.

For two nonnegative quantities $A$ and $B$,
the notation $A \lesssim B$ means that
$A \le CB$ for some unspecified constant $C>0$.
We also write $A \approx B$ if
$A \lesssim B$ and $B \lesssim A$.

We write 
$\la \xi \ra = (1+|\xi|^2)^{1/2}$ for $\xi \in \R^n$
and
$\la (\xi_1, \xi_2) \ra = (1+|\xi_1|^2 + |\xi_2|^2)^{1/2}$ for $(\xi_1, \xi_2) \in (\R^n)^2$.

For $0 < p \le \infty$,
$p'$ denotes the conjugate exponent of $p$, that is,
$p^{\prime} = \frac{p}{p-1}$ if $1< p \le \infty$ and $p^{\prime} =\infty$ if $0< p \le 1$.
%===========================================================================
%===========================================================================
\subsection{Function spaces}

%===========================================================================

Let $\phi \in \Sh(\R^n)$ be such that $\int_{\R^n}\phi(x)\, dx \neq 0$ and 
set $\phi_t(x) = t^{-n} \phi(t^{-1}x)$ for $t > 0$.
The local Hardy space $h^p = h^{p}(\R^n)$, $0< p \le \infty$,
consists of all $f \in \Sh^{\prime}(\R^n)$ such that
$\|f\|_{h^p} = \|\sup_{0 < t < 1}|\phi_t * f|\|_{L^p} < \infty$.
It is known that the space $h^p$ is independent of the choice of the function $\phi$.
It is also known that $h^p = L^p$ if $1 < p \le \infty$ 
and that  the embedding $h^1 \hookrightarrow L^1$ holds.
%===========================================================================

The space $bmo = bmo(\R^n)$ consists of all locally integrable functions $f$ on $\R^n$ 
such that
\[
\|f\|_{bmo}
=
\sup_{|Q| \le 1}
\frac{1}{|Q|}
\int_{Q}
|f(x)-f_Q|
\,
dx
+
\sup_{|Q| \ge 1}
\frac{1}{|Q|}
\int_{Q}
|f(x)|
\,
dx
<
\infty,
\]
where $f_Q = \frac{1}{|Q|} \int_Q f(x) dx$ and $Q$ runs over all cubes in $\R^n$.
It is known that the embedding $L^{\infty} \hookrightarrow bmo$ holds and 
that the dual space of $h^1$ is $bmo$.
See Goldberg \cite{Goldberg} for details on the spaces $h^p$ and $bmo$.

%===========================================================================
For $s \in \R$, we define
\[
h^p_s
=
h^p_s(\R^n)
=
\left\{
f \in \Sh^{\prime}(\R^n)
:
 \|f\|_{h^p_s} = \|(I -\Delta)^{s/2} f \|_{h^p}
 <
 \infty
\right\},
 \quad
 0 < p \le \infty,
\]
and
\[
bmo_s
=
bmo_s(\R^n)
=
\left\{
f \in \Sh^{\prime}(\R^n)
:
\|f\|_{bmo_s} = \|(I -\Delta)^{s/2} f \|_{bmo} < \infty
\right\},
\]
where $(I- \Delta)^{s/2}f = \F^{-1}[\la\cdot \ra^s \widehat{f}]$.
If $1 < p < \infty$, the space $h^p_s = L^p_s$ is a usual $L^p$-based Sobolev space.

We recall the Littlewood-Paley characterization of the space $h^p_s$.
Let $\psi_0 \in \Sh(\R^n)$ be such that 
\begin{equation}\label{supppsi0}
\supp \psi_0 \subset \{\xi \in \R^n \,:\, |\xi| \le 2\} 
\quad 
\text{and} 
\quad
\psi_0 = 1 
\quad 
\text{on} 
\quad 
\{\xi \in \R^n \,:\, |\xi| \le 1\}.
\end{equation}
We set $\psi_k(\xi) = \psi_0(2^{-k} \xi) - \psi_0( 2^{-k+1} \xi)$ , $k \ge 1$.
Then
\begin{align}
&\supp \psi_k \subset \{2^{k-1} \le |\xi| \le 2^{k+1}\}, 
\quad
k=1,2, \dots, 
\label{supppsij}
\\
&\sum_{k \ge 0} \psi_k(\xi) =1,
\quad
\xi \in \R^n.
\label{sumpsij}
\end{align}
It is known that  
the space $h^{p}_s$ can be characterized as
\[
\|f\|_{h^p_s}
\approx
\left\|
\left(
\sum_{k = 0}^{\infty}
2^{2ks}
|\psi_{k}(D)f|^2
\right)^{1/2}
\right\|_{L^p}
\]
if $0< p < \infty$. It is also known that the equivalence is independent of the choice of $\{\psi_k\}$.
For more details on this characterization, 
see e.g., Triebel \cite[Sections 2.3.8 and 2.5.8]{Triebel-ToFS}.
%===========================================================================

Next, we recall the definition of Wiener amalgam spaces. 
Let $\phi \in C_0^{\infty}(\R^n)$ be such that
\begin{equation}\label{partition-Triebel}
\left|
\sum_{\nu \in \Z^n}
\phi(\xi-\nu)
\right|
\ge
1,
\quad
\xi \in \R^n.
\end{equation}
For $0< p, q \le \infty$ and $s \in \R$, the Wiener amalgam space $W^{p,q}_s  = W^{p,q}_{s}(\R^n)$ consists of all  $f \in \Sh^{\prime}(\R^n)$ such that
\[
\|f\|_{W^{p,q}_{s}}
=
\left\|
\left(
\sum_{\nu \in \Z^n}
\la
\nu
\ra^{sq}
|\phi(D-\nu)f|^{q}
\right)^{1/q}
\right\|_{L^{p}}
<
\infty
\]
with usual modification if $q=\infty$.
We write $W_{0}^{p, q}= W^{p, q}$.
It is known that the definition of Wiener amalgam spaces does not depend on the choice of
the function $\phi$ satisyfing \eqref{partition-Triebel}.
See Triebel \cite{Triebel-Wieneramalgam} for more details on the definition of  Wiener amalgam spaces.

The inclusion relations between Wiener amalgam spaces and classical function spaces are well studied.  In particular, we will use the following embedding properties.

\begin{prop}\label{prop-inclusion-Wieneramalgam}
Let $0< p_1, p_2, p, q_1, q_2 \le \infty$ and $s \in \R$.
\begin{enumerate}
\item 
If $0<p_1 \le p_2 \le \infty$ and $0 < q_1 \le q_2 \le \infty$,
then
$
W^{p_1,q_1}_{s} \hookrightarrow W^{p_2, q_2}_{s}
$.
\item
If $0< p < \infty$, then
$
W^{p, 2}_{\alpha(p)} 
\hookrightarrow 
h^p 
\hookrightarrow 
W^{p, 2}_{\beta(p)} 
$ 
with
\[
\alpha(p)
=
\max
\left\{
0,\,
\frac{n}{2}
-
\frac{n}{p}
\right\},
\quad
\beta(p)
=
\min
\left\{
0,\,
\frac{n}{2}
-
\frac{n}{p}
\right\}.
\]
\item
$W^{\infty, 2}_{n/2} \hookrightarrow bmo \hookrightarrow W^{\infty, 2}$.
\end{enumerate}
\end{prop}
The assertion (1) is well known, and
its proof  can be found in, e.g., \cite[Proof of Lemma 2.2]{KMT-multi-2022}.
The assertion (2) was given in  \cite[Theorems 1.1 and 1.2]{CKS} and \cite[Theorem 1.2]{GWYZ}. 
Finally, the assertion (3) follows from the duality argument and
the embedding $W^{1,2} \hookrightarrow h^1 \hookrightarrow W^{1,2}_{-n/2}$.
%===========================================================================
%===========================================================================
%===========================================================================

%===========================================================================
%===========================================================================
%===========================================================================
\subsection{Key estimate}

The following lemma plays an important role to prove Theorem \ref{est-sep}.

\begin{lem}\label{lem-Bclass}
Let $a_1, a_2 \in \R$ and let 
$V(\nu_1, \nu_2)$ be one of the following functions on $\Z^n \times \Z^n$;
\[
\la \nu_1 \ra^{a_1}
\la \nu_2 \ra^{a_2},
\quad
\la \nu_1+\nu_2 \ra^{a_1}
\la \nu_2 \ra^{a_2},
\quad
\la \nu_1 \ra^{a_1}
\la \nu_1+ \nu_2 \ra^{a_2}.
\]
\begin{enumerate}
\item
If $a_1, a_2 \in \R$ satisfy 
\begin{align*}
a_1, a_2 < 0
\quad
\text{and}
\quad 
a_1+ a_2 = -n/2,
\end{align*}
then the following estimate holds for all nonegative functions $A_1$ and $A_2$ on $\Z^n$;
\begin{equation}\label{key-est*}
\left\|
\sum_{\substack{ \nu_1, \nu_2 \in \Z^n \\ \nu_1+\nu_2=\mu}}
V(\nu_1, \nu_2)
A_1(\nu_1)
A_2(\nu_2)
\right\|_{\ell^{2}_{\mu}(\Z^n)}
\lesssim
\|A_1\|_{\ell^2(\Z^n)}
\|A_2\|_{\ell^2(\Z^n)}.
\end{equation}

\item
If $a_1, a_2 \in \R$ satisfy 
\begin{align*}
a_1, a_2 \le 0
\quad
\text{and}
\quad 
a_1+ a_2 < -n/2,
\end{align*}
then the estimate \eqref{key-est*} holds for all nonegative functions $A_1$ and $A_2$ on $\Z^n$.
\end{enumerate}
\end{lem}

\begin{proof}
The assertion (1) of this lemma is proved in \cite[Lemma 2.4]{KMT-multi-2022}.
The assertion (2) is also given in \cite[Proposition 3.2]{KMT-bi-L2L2},
but we give its proof for the reader's convenience.

If $a_1, a_2 < 0$ and $a_1+a_2 < -n/2$, then the desired estimate follows from the assertion (1). 
Thus, we may assume that either  $a_1$ or $a_2$ is equal to $0$. 
By symmetry, we assume that $a_2=0$.
Then, $V(\nu_1, \nu_2) = \la \nu_1 \ra^{a_1}$ or $V(\nu_1, \nu_2) = \la \nu_1 + \nu_2 \ra^{a_1}$, and by symmetry, we only consider the former case.
By duality, it is sufficient to show that the estimate 
\begin{align*}
\sum_{\nu_1, \nu_2 \in \Z^n}
\la \nu_1 \ra^{a_1}
A_0(\nu_1+\nu_2)
A_1(\nu_1)
A_2(\nu_2)
\lesssim
\|A_0\|_{\ell^2}
\|A_1\|_{\ell^2}
\|A_2\|_{\ell^2}
\end{align*}
holds for all nonnegative functions $A_0, A_1, A_2 \in \ell^2(\Z^n)$.
If $a_2 = 0$, then $a_1 < -n/2$, 
and hence we have by the Schwarz inequality
\begin{align*}
\sum_{\nu_1, \nu_2 \in \Z^n}
\la \nu_1 \ra^{a_1}
A_0(\nu_1+\nu_2)
A_1(\nu_1)
A_2(\nu_2)
&\le
\sum_{\nu_1 \in \Z^n}
\la \nu_1 \ra^{a_1}
A_1(\nu_1)
\|A_0(\nu_1+\nu_2)\|_{\ell^2_{\nu_2}}
\|A_2(\nu_2)\|_{\ell^2_{\nu_2}}
\\
&\lesssim
\|A_0\|_{\ell^2}
\|A_1\|_{\ell^2}
\|A_2\|_{\ell^2}.
\end{align*}
The proof of Lemma \ref{lem-Bclass} is complete.
\end{proof}

\section{Bilinear pseudo-differential operators of $S_{0,0}$-type}
\label{secS00type}
To prove Theorem \ref{main-thm}, 
we will use the following  three $S_{0,0}$-type symbol classes.

\begin{dfn}
Let $m_1, m_2 \in \R$.
\begin{enumerate}
\item
The class
$BS^{(m_1, m_2)}_{0,0}$ 
denotes the set of all smooth functions
$\sigma = \sigma(x, \xi_1,\xi_2)$ on $(\R^n)^3$
that satisfy
\[
|
\partial_x^{\alpha}\partial_{\xi_1}^{\beta_1}\partial_{\xi_2}^{\beta_2}
\sigma(x, \xi_1, \xi_2)
|
\le
C_{\alpha, \beta_1, \beta_2}
(1+|\xi_1|)^{m_1}
(1+|\xi_2|)^{m_2}
\]
for all $\alpha, \beta_1, \beta_2 \in \N_0^n$.
\item
The classes
$BS^{(m_1, m_2), *1}_{0,0}$ and  $BS^{(m_1, m_2), *2}_{0,0}$
denote the sets of all smooth functions
$\sigma = \sigma(x, \xi_1,\xi_2)$ on $(\R^n)^3$
that satisfy
\[
|
\partial_x^{\alpha}\partial_{\xi_1}^{\beta_1}\partial_{\xi_2}^{\beta_2}
\sigma(x, \xi_1, \xi_2)
|
\le
C_{\alpha, \beta_1, \beta_2}
(1+|\xi_1+\xi_2|)^{m_1}
(1+|\xi_2|)^{m_2}
\]
and
\[
|
\partial_x^{\alpha}\partial_{\xi_1}^{\beta_1}\partial_{\xi_2}^{\beta_2}
\sigma(x, \xi_1, \xi_2)
|
\le
C_{\alpha, \beta_1, \beta_2}
(1+|\xi_1|)^{m_1}
(1+|\xi_1+\xi_2|)^{m_2}
\]
for all $\alpha, \beta_1, \beta_2 \in \N_0^n$, respectively.
\end{enumerate}
\end{dfn}
\noindent
Notice that 
$BS^m_{0,0} \subset BS^{(m_1, m_2)}_{0,0}$ 
and 
$BS^m_{0,0} \subset BS^{(m_1, m_2), *j}_{0,0}$, $j=1,2$,
if $m_1, m_2 \le 0$ and $m_1+m_2=m$.

We consider the boundedness of bilinear pseudo-differential operators with symbols in the above classes.
The following theorem plays an essential role to prove Theorem \ref{main-thm}.

\begin{thm}
\label{est-sep}
Let $0 <p_1, p_2, p \le \infty$, $1/p \le 1/p_1 + 1/p_2$.
We assume that $m_1$ and $m_2$ satisfy
\begin{equation}\label{condition-m1m2}
m_1+m_2 
=
\min
\left\{
\frac{n}{p}, 
\frac{n}{2}
\right\}
-
\max
\left\{
\frac{n}{p_1},
\frac{n}{2}
\right\}
-
\max
\left\{
\frac{n}{p_2},
\frac{n}{2}
\right\}.
\end{equation}

\begin{enumerate}
\item
If  $0< p < \infty$,
and
if $m_1$ and $m_2$ satisfy
\begin{align}\label{condition-mj}
\begin{split}
&
-
\max
\left\{
\frac{n}{p_1},
\frac{n}{2}
\right\}
<
m_1
<
\min
\left\{
\frac{n}{p}, 
\frac{n}{2}
\right\}
-
\max
\left\{
\frac{n}{p_1},
\frac{n}{2}
\right\},
\\
&
-
\max
\left\{
\frac{n}{p_2},
\frac{n}{2}
\right\}
<
m_2
<
\min
\left\{
\frac{n}{p}, 
\frac{n}{2}
\right\}
-
\max
\left\{
\frac{n}{p_2},
\frac{n}{2}
\right\},
\end{split}
\end{align}
then 
$
\Op(BS^{(m_1, m_2)}_{0, 0})
\subset
B(h^{p_1} \times h^{p_2} \to h^p).
$
\item
If $1< p_1 \le \infty$,
and if $m_1$ and $m_2$ satisfy
\begin{align}\label{assumption-dual1}
\begin{split}
&
-\max
\left\{
\frac{n}{p^{\prime}}, 
\frac{n}{2}
\right\}
<
m_1
<
\min
\left\{
\frac{n}{p_1^\prime}, 
\frac{n}{2}
\right\}
-\max
\left\{
\frac{n}{p^{\prime}}, 
\frac{n}{2}
\right\},
\\
&
-\max
\left\{
\frac{n}{p_2}, 
\frac{n}{2}
\right\}
<
m_2
<
\min
\left\{
\frac{n}{p_1^\prime}, 
\frac{n}{2}
\right\}
-\max
\left\{
\frac{n}{p_2}, 
\frac{n}{2}
\right\},
\end{split}
\end{align}
then 
$
\Op(BS^{(m_1, m_2), *1}_{0, 0})
\subset
B(L^{p_1} \times h^{p_2} \to h^p).
$
\item
If $1< p_2 \le \infty$,
and if $m_1$ and $m_2$ satisfy
\begin{align}\label{assumption-dual2}
\begin{split}
&
-\max
\left\{
\frac{n}{p_1}, 
\frac{n}{2}
\right\}
<
m_1
<
\min
\left\{
\frac{n}{p_2^\prime}, 
\frac{n}{2}
\right\}
-\max
\left\{
\frac{n}{p_1}, 
\frac{n}{2}
\right\},
\\
&
-\max
\left\{
\frac{n}{p^{\prime}}, 
\frac{n}{2}
\right\}
<
m_2
<
\min
\left\{
\frac{n}{p_2^\prime}, 
\frac{n}{2}
\right\}
-\max
\left\{
\frac{n}{p^\prime}, 
\frac{n}{2}
\right\},
\end{split}
\end{align}
then 
$\Op(BS^{(m_1, m_2), *2}_{0, 0}) \subset B(h^{p_1} \times L^{p_2} \to h^p)$.
\end{enumerate}
Here $h^p$ (resp. $h^{p_1}$ ($L^{p_1}$), $h^{p_2}$ ($L^{p_2}$)) should be replaced by $bmo$ if 
$p= \infty$ 
(resp. $p_1=\infty$, $p_2= \infty$).
\end{thm}

\begin{rem}
We will give some remarks on the above thorem.
\begin{enumerate}
\renewcommand{\labelenumi}{(\roman{enumi})}
\item
The assertion (1) of Theorem \ref{est-sep} 
was first proved in 
\cite[Theorem 1.5]{KMT-multi-2022} in the multilinear setting.
In \cite{KMT-multi-2022},
they proved the boundedness under the assumption 
\eqref{condition-m1m2} 
and
\begin{equation}\label{condition-mj-KMT}
-
\max
\left\{
\frac{n}{p_j},
\frac{n}{2}
\right\}
<
m_j
<
\frac{n}{2}
-
\max
\left\{
\frac{n}{p_j},
\frac{n}{2}
\right\},
\quad
j=1,2.
\end{equation}
However, we see that the condition \eqref{condition-mj} is equivalent to \eqref{condition-mj-KMT} 
under the assumption \eqref{condition-m1m2}.
See the following figure:

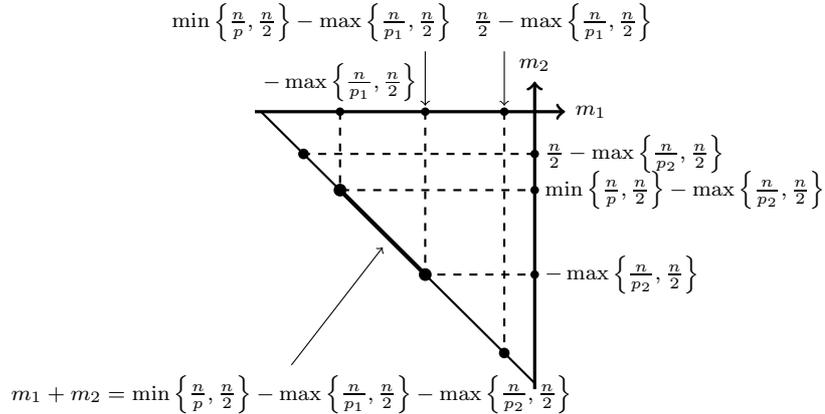
\begin{center}
\begin{tikzpicture}[scale=0.8]
\coordinate (O) at (0, 2.5);
\coordinate (A1) at (-3.2, 2.5);
\coordinate (B1) at (0, -0.2);
\coordinate (A2) at (-1.8, 2.5);
\coordinate (C2*) at (-0.5, 2.5);
\coordinate (C2**) at (-0.5, 3.5);
\coordinate (C2***) at (-0.5, -1.5);
\coordinate (C2****) at (-0.5, 2.6);
\coordinate (A2*) at (-1.8, 3.5);
\coordinate (A2**) at (-1.8, 2.6);
\coordinate (B2) at (0, 1.2);
\coordinate (D2) at (0, 1.8);
\coordinate (D2***) at (-3.8, 1.8);
\coordinate (A1B2) at (-3.2, 1.2);
\coordinate (A2B1) at (-1.8, -0.2);

\fill (C2***) circle (2.5pt);
\fill (D2***) circle (2.5pt);
\fill (A1B2) circle (3pt);
\fill (A2B1) circle (3pt);

\fill (A1) circle (2pt) node[anchor=south]
{\scriptsize$-\max\left\{\frac{n}{p_1}, \frac{n}{2}\right\} $};

\fill (B1) circle (2pt) node[anchor=west]
{\scriptsize$-\max\left\{\frac{n}{p_2}, \frac{n}{2} \right\}$};

\fill(A2) circle (2pt);
\draw[->] (A2*) -- (A2**);
\draw (A2*) node [anchor=south]
{\scriptsize$\min\left\{\frac{n}{p}, \frac{n}{2}\right\}-\max\left\{\frac{n}{p_1}, \frac{n}{2} \right\}$\qquad\qquad\qquad\qquad\qquad\qquad};

\fill(B2) circle (2pt) node[anchor=west]
{\scriptsize$\min \left\{\frac{n}{p}, \frac{n}{2}\right\}-\max\left\{\frac{n}{p_2}, \frac{n}{2} \right\}$};

\fill(C2*) circle (2pt);
\draw[->] (C2**) -- (C2****);
\draw (C2**) node [anchor=south]
{\qquad \qquad \scriptsize$\frac{n}{2}-\max\left\{\frac{n}{p_1}, \frac{n}{2} \right\}$};

\fill (D2) circle (2pt) node[anchor=west]
{\scriptsize$\frac{n}{2}-\max\left\{\frac{n}{p_2}, \frac{n}{2} \right\}$};

\draw[thick] (-4.5, 2.5) -- (0, -2);
\draw[->] (-4, -1.7) -- (-2.5, 0.25);
\draw (-4, -1.7)  node[anchor= north]
{\scriptsize $m_1+m_2=\min \left\{\frac{n}{p}, \frac{n}{2}\right\}-\max\left\{\frac{n}{p_1}, \frac{n}{2} \right\} -\max\left\{\frac{n}{p_2}, \frac{n}{2} \right\}$};

\draw[dashed, thick]
(A1) -- (A1B2)-- (B2);

\draw[dashed,  thick]
(A2) -- (A2B1) -- (B1);

\draw[dashed,  thick] (C2*) -- (C2***);
\draw[dashed,  thick] (D2) -- (D2***);

\draw[ultra thick] (A1B2) -- (A2B1);

\draw[->, very thick] (-4.6, 2.5) -- (0.5, 2.5);
\draw (0.5, 2.5)  node[anchor= west]{\scriptsize $m_1$};
\draw[->, very thick] (0, -2.1) -- (0, 3);
\draw (0, 3)  node[anchor= south]{\scriptsize $m_2$};

\end{tikzpicture}
\end{center}

\item
The condition \eqref{condition-mj} says that
we need to assume that $p \neq \infty$ for the assertion (1).
Similarly, the condition \eqref{assumption-dual1} and \eqref{assumption-dual2} yield that we need the condition $p_1 > 1$ 
for the assertion (2) 
and the condition $p_2 > 1$
for the assertion (3),
respectively.
\end{enumerate}
\end{rem}

To prove Theorem \ref{est-sep}, we will use the following partition of unity given in \cite{PS}. 

\begin{lem}[{\cite[Lemma 2.2]{PS}}]\label{lem-decomp}
For each $L \in \N$,  there exists a sequence of Schwartz functions 
$\{\chi_{\ell}\}_{\ell \in \Z^n}$ such that
\begin{align*}
\supp \chi_{\ell} \subset \ell + [-1, 1]^n,
\quad
\sup_{\ell \in \Z^n} \|\F^{-1}\chi_{\ell}\|_{L^1} < \infty,
\\
\sum_{\ell \in \Z^n}
\la \ell \ra^{-2L}
\la \xi \ra^{2L}
\chi_{\ell}(\xi)
=
1
\quad
\text{for all}
\quad
\xi \in \R^n.
\end{align*}
\end{lem}

\begin{proof}[Proof of Theorem \ref{est-sep}]
The basic ideas of the proof  below go back to \cite[Proof of Theorem 1.5]{KMT-multi-2022}. 
The assertion (1) is proved in \cite[Theorem 1.5]{KMT-multi-2022}.  
We shall give the proof of the assertion (2). The assertion (3) can be shown in the same way as below.

We define the number $p_0$ by $1/p_0 = 1/p_1+1/p_2$.
In order to obtain the desired boundedness, 
it is sufficient to prove that
\begin{equation}\label{est-sharp}
\Op(BS^{(m_1, m_2), *1}_{0,0})
\subset
B(W^{p_1, 2}_{\beta(p_1)} \times W^{p_2, 2}_{\beta(p_2)} 
\to 
W^{p_0, 2}_{\alpha(p)})
\end{equation}
under the assumption \eqref{assumption-dual1}.
In fact,  by Proposition \ref{prop-inclusion-Wieneramalgam},  we have the following embeddings;
\begin{align*}
&
h^{p_j} 
\hookrightarrow 
W^{p_j, 2}_{\beta(p_j)}
\quad
\text{if}
\quad
0< p_j < \infty,
\quad
j=1,2,
\\
&W^{p_0, 2}_{\alpha(p)}
\hookrightarrow
W^{p, 2}_{\alpha(p)}
\hookrightarrow
h^{p}
\quad
\text{if}
\quad
0< p < \infty,
\\
&
W^{p_0, 2}_{n/2}
\hookrightarrow
W^{\infty, 2}_{n/2} 
\hookrightarrow 
bmo
\hookrightarrow
W^{\infty, 2}.
\end{align*}
In order to prove \eqref{est-sharp}, 
we use the Fourier series expansion.
This method was used by Coifman and Meyer in \cite{CM, CM-2}.

Let $\vphi \in \Sh(\R^n)$ be such that
\begin{equation*}
\supp \vphi 
\subset 
[-1,1]^n,
\quad
\sum_{\nu \in \Z^n}
\vphi(\xi -\nu)
=
1, 
\quad
\xi \in \R^n.
\end{equation*}
Then, we decompose the symbol $\sigma$ as
\[
\sigma(x, \xi_1, \xi_2)
=
\sum_{(\nu_1, \nu_2) \in (\Z^n)^2}
\sigma(x, \xi_1, \xi_2)
\vphi(\xi_1-\nu_1)
\vphi(\xi_2-\nu_2)
=
\sum_{\Nu \in (\Z^n)^2}
\sigma_{\Nu}(x, \xi_1, \xi_2),
\]
where $\sigma_{\Nu}(x, \xi_1, \xi_2) = \sigma(x, \xi_1, \xi_2)
\vphi(\xi_1-\nu_1)
\vphi(\xi_2-\nu_2)$.
If we set
\[
\widetilde{\sigma}_{\Nu}(x, \xi_1, \xi_2)
=
\sum_{\mu_1, \mu_2 \in \Z^n}
\sigma_{\Nu}(x, \xi_1-2\pi\mu_1, \xi_2-2\pi\mu_2),
\]
then $\widetilde{\sigma}_{\Nu}$ is a $2\pi\Z^{2n}$-periodic function
with respect to the variables $(\xi_1, \xi_2)$.
Let $\widetilde{\vphi} \in \Sh(\R^n)$ be such that
$\supp \widetilde{\vphi} \subset [-3, 3]^n$,
$0 \le \widetilde{\vphi} \le 1$, 
and
$\widetilde{\vphi}=1$ on $[-1,1]^n$.
Since
$\widetilde{\vphi} \vphi = 1$ 
and 
$\widetilde{\sigma} = \sigma_{\Nu}$ if $(\xi_1, \xi_2) \in \Nu+[-3, 3]^{2n}$,
we have
\begin{align*}
\sigma_{\Nu}(x, \xi_1, \xi_2)
=
\sigma_{\Nu}(x, \xi_1, \xi_2)
\widetilde{\vphi}(\xi_1-\nu_1)
\widetilde{\vphi}(\xi_2-\nu_2)
=
\widetilde{\sigma}_{\Nu}(x, \xi_1, \xi_2)
\widetilde{\vphi}(\xi_1-\nu_1)
\widetilde{\vphi}(\xi_2-\nu_2).
\end{align*}
Thus, the Fourier series expansion of $\widetilde{\sigma}_{\Nu}$ gives 
\[
\sigma_{\Nu}(x, \xi_1, \xi_2)
=
\sum_{\K =(k_1, k_2) \in (\Z^n)^2}
e^{i(\xi_1 \cdot k_1+\xi_2 \cdot k_2)}
P_{\Nu, \K}(x)
\widetilde{\vphi}(\xi_1-\nu_1)
\widetilde{\vphi}(\xi_2-\nu_2),
\]
where
\[
P_{\Nu, \K}(x)
=
\frac{1}{(2\pi)^{2n}}
\int_{\Nu+[-3,3]^{2n}}
e^{-i(\eta_1 \cdot k_1+\eta_2 \cdot k_2)}
\sigma_{\Nu}(x, \eta_1, \eta_2)
\,
d\eta_1d\eta_2.
\]
Furthermore, integration by parts gives that
$P_{\Nu, \K}(x)=\la (k_1, k_2) \ra^{-2K} Q_{\Nu, \K}(x)$ with $K \in \N$ and
\[
Q_{\Nu, \K}(x)
=
\frac{1}{(2\pi)^{2n}}
\int_{\Nu+[-3,3]^{2n}}
e^{-i(\eta_1\cdot k_1+\eta_2 \cdot k_2)}
(I-\Delta_{\eta_1, \eta_2})^{K}
[
\sigma_{\nu}(x, \eta_1, \eta_2)
]
\,
d\eta_1d\eta_2.
\]
Applying Lemma \ref{lem-decomp}, we have
\[
Q_{\Nu, \K}
=
\sum_{\ell \in \Z^n}
\la \ell \ra^{-2L}
\F^{-1}[\chi_{\ell} \la \cdot\ra^{2L} \F[Q_{\Nu, \K}]]
=
\sum_{\ell \in \Z^n}
\la \ell \ra^{-2L}
Q_{\Nu, \K, \ell}
\]
with $Q_{\Nu, \K, \ell} = \F^{-1}[\chi_{\ell} \la \cdot\ra^{2L} \F Q_{\Nu, \K}] = (\F^{-1}\chi_{\ell})* (I-\Delta)^{L}Q_{\Nu, \K}$.
Combining these decomposition, we obtain
\begin{align*}
&T_\sigma(f_1, f_2)(x)
=
\sum_{\K \in (\Z^n)^2}
\sum_{\ell \in \Z^n}
\la (k_1, k_2) \ra^{-2K}
\la \ell \ra^{-2L}
\sum_{\Nu \in (\Z^n)^2}
Q_{\Nu, \K, \ell}(x)
\,
F^1_{\nu_1, k_1}(x)
F^2_{\nu_2, k_2}(x),
\end{align*}
where
$F^j_{\nu_j, k_j}(x)=\widetilde{\vphi}(D-\nu_j)f_{j}(x+k_j)$, $j=1, 2$.

Next, we shall prove the estimate
\begin{equation}\label{Q-est}
\sup_{\K, \ell}
\|Q_{\Nu, \K, \ell}\|_{L^{\infty}}
\lesssim
\la \nu_1+\nu_2 \ra^{m_1}
\la \nu_2 \ra^{m_2}.
\end{equation}
If $\sigma \in BS^{(m_1, m_2), *1}_{0,0}$, then 
\[
|\pa_{x}^{\alpha} \pa_{\xi_1}^{\beta_1} \pa_{\xi_2}^{\beta_2} \sigma_{\Nu}(x, \xi_1, \xi_2)|
\lesssim
\la \nu_1+\nu_2 \ra^{m_1}
\la \nu_2 \ra^{m_2},
\]
and this yields that
\[
\|(I-\Delta)^{L}Q_{\Nu, \K}\|_{L^{\infty}}
\lesssim
\la \nu_1+\nu_2 \ra^{m_1}
\la \nu_2 \ra^{m_2}.
\]
Hence, by Young's inequality,
\[
|Q_{\Nu, \K, \ell}(x)|
\le
\|\F^{-1}\chi_{\ell}\|_{L^1}
\|(I-\Delta)^{L}Q_{\Nu, \K}\|_{L^{\infty}}
\lesssim
\la \nu_1+\nu_2 \ra^{m_1}
\la \nu_2 \ra^{m_2}.
\]
Here we remark that the implicit constant does not depend on $\ell$, $\K$ and $x$.
Thus, we obtain \eqref{Q-est}.

Now, we set $\widetilde{p}_0 = \min \{1, p_0\}$. 
Taking the integer $K$ satisfying $K > 2n$, 
we have
\begin{align*}
&\|T_\sigma(f_1, f_2)\|_{W^{p_0, 2}_{\alpha(p)}}^{\widetilde{p}_0}
\\
&\lesssim
\sup_{\K}
\sum_{\ell \in \Z^n}
\la \ell \ra^{-L \widetilde{p}_0}
\left\|
\left\|
\la \mu \ra^{\alpha(p)}
\phi(D-\mu)
\left[
\sum_{\Nu}
Q_{\Nu, \K, \ell}(x)
F^1_{\nu_1, k_1}(x)
F^2_{\nu_2, k_2}(x)
\right]
\right\|_{\ell^2_{\mu}}
\right\|_{L^{p_0}_x}^{\widetilde{p}_0}.
\end{align*}
Since 
$\supp \F Q_{\Nu, \K, \ell}  \subset \ell + [-1,1]^n$ and 
$\supp \F F^j_{\nu_j, k_j}  \subset \nu_j + [-1,1]^n$, $j=1,2$,
and since the function $\phi$ (which is used in the definition of Wiener amalgam spaces) has compact support, we have
\begin{align*}
&\left\|
\left\|
\la \mu \ra^{\alpha(p)}
\phi(D-\mu)
\left[
\sum_{\Nu}
Q_{\Nu, \K, \ell}
F^1_{\nu_1, k_1}
F^2_{\nu_2, k_2}
\right]
\right\|_{\ell^2_{\mu}}
\right\|_{L^{p_0}}^{\widetilde{p}_0}
\\
&=
\left\|
\left\|
\sum_{\Nu: |\nu_1+\nu_2+\ell -\mu| \lesssim 1}
\la \mu \ra^{\alpha(p)}
\phi(D-\mu)
[
Q_{\Nu, \K, \ell}
F^1_{\nu_1, k_1}
F^2_{\nu_2, k_2}
]
\right\|_{\ell^2_{\mu}}
\right\|_{L^{p_0}}^{\widetilde{p}_0}
\\
&\lesssim
\sum_{\tau: |\tau| \lesssim 1}
\left\|
\left\|
\sum_{\Nu: \nu_1+\nu_2+\ell -\mu =\tau}
\la \mu \ra^{\alpha(p)}
\phi(D-\mu)
[
Q_{\Nu, \K, \ell}
F^1_{\nu_1, k_1}
F^2_{\nu_2, k_2}
]
\right\|_{\ell^2_{\mu}}
\right\|_{L^{p_0}}^{\widetilde{p}_0}
\\
&=
\sum_{\tau: |\tau| \lesssim 1}
\left\|
\left\|
\la \mu-\tau + \ell \ra^{\alpha(p)}
\phi(D-\mu+\tau-\ell)
\left[
\sum_{\Nu: \nu_1+\nu_2=\mu}
Q_{\Nu, \K, \ell}
F^1_{\nu_1, k_1}
F^2_{\nu_2, k_2}
\right]
\right\|_{\ell^2_{\mu}}
\right\|_{L^{p_0}}^{\widetilde{p}_0}
\\
&\lesssim
\la \ell \ra^{\alpha(p) \widetilde{p}_0}
\sum_{\tau: |\tau| \lesssim 1}
\left\|
\left\|
\la \mu \ra^{\alpha(p)}
\phi(D-\mu+\tau-\ell)
\left[
\sum_{\Nu: \nu_1+\nu_2=\mu}
Q_{\Nu, \K, \ell}
F^1_{\nu_1, k_1}
F^2_{\nu_2, k_2}
\right]
\right\|_{\ell^2_{\mu}}
\right\|_{L^{p_0}}^{\widetilde{p}_0},
\end{align*}
where we used the inequality
$\la \mu-\tau + \ell \ra^{\alpha(p)} \lesssim \la \mu \ra^{\alpha(p)} \la \tau \ra^{\alpha(p)} \la \ell \ra^{\alpha(p)}$ in the last inequality.
Taking the integer $L$ satisfying 
$(-L+\alpha(p))\widetilde{p}_0 < -n$, 
we obtain
\begin{align*}
&\|T_\sigma(f_1, f_2)\|_{W^{p_0, 2}_{\alpha(p)}}^{\widetilde{p}_0}
\\
&\lesssim
\sup_{\K}
\sum_{\ell \in \Z^n}
\la \ell \ra^{(-L + \alpha(p))\widetilde{p}_0}
\sum_{\tau: |\tau| \lesssim 1}
\left\|
\left\|
\la \mu \ra^{\alpha(p)}
\phi(D-\mu+\tau-\ell)
F_{\mu, \K, \ell}
\right\|_{\ell^2_{\mu}}
\right\|_{L^{p_0}}^{\widetilde{p}_0}
\\
&\lesssim
\sup_{\K, \ell}
\sum_{\tau: |\tau| \lesssim 1}
\left\|
\left\|
\la \mu \ra^{\alpha(p)}
\phi(D-\mu+\tau-\ell)
F_{\mu, \K, \ell}
\right\|_{\ell^2_{\mu}}
\right\|_{L^{p_0}}^{\widetilde{p}_0}
\\
&=
(\dag),
\end{align*}
where
\[
F_{\mu, \K, \ell}
=
\sum_{\Nu: \nu_1+\nu_2=\mu}
Q_{\Nu, \K, \ell}
F^1_{\nu_1, k_1}
F^2_{\nu_2, k_2}.
\]
Then we have $\supp \F F_{\mu, \K, \ell} \subset \{|\zeta-\mu-\ell| \lesssim 1\}$.
Since $\supp \phi$ is compact, we have
\[
\supp \F[\Phi(\cdot) F_{\mu, \K, \ell}(x-\cdot)]
\subset
\{
\zeta 
:
|\zeta-\mu-\ell|
\lesssim
1
\},
\]
where $\Phi = \F^{-1} \phi$.
Hence, by Nikol'skij inequality (see, e.g., \cite[Proposition 1.3.2]{Triebel-ToFS}), we have
\[
|\phi(D-\mu+\tau-\ell)F_{\mu, \K, \ell}(x)|
\le
\|\Phi(y)F_{\mu, \K, \ell}(x-y)\|_{L^1_y}
\lesssim
\|\Phi(y)F_{\mu, \K, \ell}(x-y)\|_{L^{\widetilde{p}_0}_y}.
\]
Thus,  Minkowski's inequality  gives
\begin{align*}
(\dag)
&\lesssim
\sup_{\K, \ell}
\left\|
\left\|
\la \mu \ra^{\alpha(p)}
\|\Phi(y)F_{\mu, \K, \ell}(x-y)
\|_{L^{\widetilde{p}_0}_y}
\right\|_{\ell^2_{\mu}}
\right\|_{L^{p_0}_x}
\\
&\le
\sup_{\K, \ell}
\left\|
\left\|
\left\|
\la \mu \ra^{\alpha(p)}
\Phi(y)F_{\mu, \K, \ell}(x-y)
\right\|_{\ell^2_{\mu}}
\right\|_{L^{p_0}_x}
\right\|_{L^{\widetilde{p}_0}_y}
\\
&\lesssim
\sup_{\K, \ell}
\left\|
\left\|
\la \mu \ra^{\alpha(p)}
F_{\mu, \K, \ell}
\right\|_{\ell^2_{\mu}}
\right\|_{L^{p_0}}.
\end{align*}
Therefore,  combining this with \eqref{Q-est}, we obtain
\begin{align*}
&\|T_\sigma(f_1, f_2)\|_{W^{p_0, 2}_{\alpha(p)}}
\\
&\lesssim
\sup_{\K}
\left\|
\left\|
\sum_{\nu_1, \nu_2 \in \Z^n : \nu_1+\nu_2=\mu}
\la \nu_1 + \nu_2 \ra^{m_1+\alpha(p)}
\la \nu_2 \ra^{m_2}
\left|
F^1_{\nu_1, k_1}
\right|
\left|
F^2_{\nu_2, k_2}
\right|
\right\|_{\ell^2_{\mu}}
\right\|_{L^{p_0}}
\\
&=
\sup_{\K}
\left\|
\left\|
\sum_{\nu_1, \nu_2 \in \Z^n : \nu_1+\nu_2=\mu}
\la \nu_1 + \nu_2 \ra^{m_1+\alpha(p)}
\la \nu_1 \ra^{-\beta(p_1)}
\la \nu_2 \ra^{m_2-\beta(p_2)}
\left|
\widetilde{F}^1_{\nu_1, k_1}
\right|
\left|
\widetilde{F}^2_{\nu_2, k_2}
\right|
\right\|_{\ell^2_{\mu}}
\right\|_{L^{p_0}}
\\
&=
(\ddag),
\end{align*}
where $\widetilde{F}^j_{\nu_j, k_j} = \la \nu_j \ra^{\beta(p_j)} F^j_{\nu_j, k_j}$, $j=1,2$.
We divide the sum into the following three parts:
\begin{align*}
\sum_{\nu_1, \nu_2 \in \Z^n : \nu_1+\nu_2=\mu}
(\cdots)
&=
\left(
\sum_{\substack{\nu_1, \nu_2 \in \Z^n: \nu_1+\nu_2=\mu \\ |\nu_1| > 2|\nu_2|}}
+
\sum_{\substack{\nu_1, \nu_2 \in \Z^n: \nu_1+\nu_2=\mu \\ |\nu_2|/2 \le |\nu_1| \le 2|\nu_2|}}
+
\sum_{\substack{\nu_1, \nu_2 \in \Z^n: \nu_1+\nu_2=\mu \\ 2|\nu_1| < |\nu_2|}}
\right)
(\cdots)
\\
&=:
S_1 + S_2 + S_3.
\end{align*}
Then
\[
(\ddag)
\lesssim
\sup_{\K}
\left\|
\left\|
S_1
\right\|_{\ell^2_{\mu}}
\right\|_{L^{p_0}}
+
\sup_{\K}
\left\|
\left\|
S_2
\right\|_{\ell^2_{\mu}}
\right\|_{L^{p_0}}
+
\sup_{\K}
\left\|
\left\|
S_3
\right\|_{\ell^2_{\mu}}
\right\|_{L^{p_0}}
=:
\widetilde{S}_1 + \widetilde{S}_2 + \widetilde{S}_3.
\]

\noindent
\textit{Estimate for $\widetilde{S}_1$}:
Since $\la \nu_1+\nu_2 \ra \approx \la \nu_1\ra$ if $|\nu_1| >2|\nu_2|$,
we have 
\begin{equation}\label{equiv1}
\widetilde{S}_1
\lesssim
\sup_{\K}
\left\|
\left\|
\sum_{\Nu: \nu_1+\nu_2=\mu}
\la \nu_1 \ra^{m_1+\alpha(p)-\beta(p_1)}
\la \nu_2 \ra^{m_2-\beta(p_2)}
\left|
\widetilde{F}^1_{\nu_1, k_1}
\right|
\left|
\widetilde{F}^2_{\nu_2, k_2}
\right|
\right\|_{\ell^2_{\mu}}
\right\|_{L^{p_0}}.
\end{equation}
Furthermore, our assumption 
\eqref{condition-m1m2}
and
\eqref{assumption-dual1} 
imply that
\begin{align*}
&
(m_1+\alpha(p)-\beta(p_1)) + (m_2-\beta(p_2))
=
-\frac{n}{2}
\\
&
-\frac{n}{2} < m_1+\alpha(p)-\beta(p_1) < 0,
\quad
-\frac{n}{2} < m_2 -\beta(p_2) < 0
\end{align*}
since 
\[
\min\left\{\frac{n}{p}, \frac{n}{2}\right\} = -\alpha(p) + \frac{n}{2},
\quad
\max\left\{\frac{n}{p_j}, \frac{n}{2}\right\} = -\beta(p_j) + \frac{n}{2},
\quad
j=1,2,
\]
and since $\beta(p_1) \le 0$, $\min\{n/p_1^{\prime}, n/2\} \le n/2$ and
\begin{align*}
\min\left\{\frac{n}{p_1^{\prime}}, \frac{n}{2}\right\} = \beta(p_1) + \frac{n}{2},
\quad
\max\left\{\frac{n}{p^{\prime}}, \frac{n}{2}\right\} = \alpha(p) + \frac{n}{2}.
\end{align*}
Hence, by the assertion (1) of Lemma \ref{lem-Bclass} and  H\"older's inequality, 
we obtain
\begin{align*}
\widetilde{S}_1
\lesssim
\sup_{\K}
\left\|
\left\|
\widetilde{F}^1_{\nu_1, k_1}
\right\|_{\ell^2_{\nu_1}}
\left\|
\widetilde{F}^2_{\nu_2, k_2}
\right\|_{\ell^2_{\nu_2}}
\right\|_{L^{p_0}}
\lesssim
\sup_{\K}
\left\|
\left\|
\widetilde{F}^1_{\nu_1, k_1}
\right\|_{\ell^2_{\nu_1}}
\right\|_{L^{p_1}}
\left\|
\left\|
\widetilde{F}^2_{\nu_2, k_2}
\right\|_{\ell^2_{\nu_2}}
\right\|_{L^{p_2}}
\end{align*}
Since the function $\widetilde{\varphi}$ satisfies \eqref{partition-Triebel}, 
we have the following equivalence:
\[
\left\|
\left\|
\widetilde{F}^j_{\nu_j, k_j}
\right\|_{\ell^2_{\nu_j}}
\right\|_{L^{p_j}}
=
\left\|
\left\|
\la \nu_j \ra^{\beta(p_j)}
\widetilde{\varphi}(D-\nu_j)f_j(x+k_j)
\right\|_{\ell^2_{\nu_j}}
\right\|_{L^{p_j}_x}
\approx
\|f_j\|_{W^{p_j, 2}_{\beta(p_j)}},
\quad
j=1,2.
\]
Thus, we obtain the desired estimate.

\bigskip
\noindent
\textit{Estimate for $\widetilde{S}_2$}:
Since $\la \nu_1 \ra \approx \la \nu_2\ra$, 
it follows that
\begin{equation*}
\widetilde{S}_2
\lesssim
\sup_{\K}
\left\|
\left\|
\sum_{\Nu: \nu_1+\nu_2=\mu}
\la \nu_1 + \nu_2 \ra^{m_1+\alpha(p)}
\la \nu_2 \ra^{m_2-\beta(p_2)-\beta(p_1)}
\left|
\widetilde{F}^1_{\nu_1, k_1}
\right|
\left|
\widetilde{F}^2_{\nu_2, k_2}
\right|
\right\|_{\ell^2_{\mu}}
\right\|_{L^{p_0}}.
\end{equation*}
Furthermore, we have 
$(m_1+\alpha(p)) + (m_2-\beta(p_2)-\beta(p_1))= -n/2$ and
\begin{align*}
-\frac{n}{2} < m_1+\alpha(p) < 0,
\quad
-\frac{n}{2} < m_2 -\beta(p_2)-\beta(p_1) < 0
\end{align*}
under the assumption \eqref{assumption-dual1}.
Thus, by Lemma \ref{lem-Bclass},  we obtain the desired estimate
in the same way as above.

\bigskip
\noindent
\textit{Estimate for $S_3$}:
For $\nu_1, \nu_2 \in \Z^n$ satisfying $2|\nu_1|< |\nu_2|$, we have $\la \nu_1+ \nu_2 \ra \approx \la \nu_2 \ra \gtrsim \la \nu_1\ra$. 
Hence $\widetilde{S}_3$ is estimated by the right hand side of \eqref{equiv1} since $m_1 + \alpha(p) < 0$,
and consequently we also obtain the desired estimate in the same way as above.

\bigskip
The proof of Theorem \ref{est-sep} is complete.
\end{proof}

In the same way as above, we can prove the following  proposition 
by using the assertion (2) of Lemma \ref{lem-Bclass}.

\begin{prop}\label{sc-prop}
Let $0 <p_1, p_2, p \le \infty$ and $1/p \le 1/p_1 + 1/p_2$.
We assume that $m_1$ and $m_2$ satisfy
\begin{align*}
m_1+m_2 
< 
\min
\left\{
\frac{n}{p}, 
\frac{n}{2}
\right\}
-
\max
\left\{
\frac{n}{p_1},
\frac{n}{2}
\right\}
-
\max
\left\{
\frac{n}{p_2},
\frac{n}{2}
\right\}.
\end{align*}

\begin{enumerate}
\item
If $m_1$ and $m_2$ satisfy
\begin{align*}
\qquad
m_1
\le
\min
\left\{
\frac{n}{p}, 
\frac{n}{2}
\right\}
-
\max
\left\{
\frac{n}{p_1},
\frac{n}{2}
\right\},
\quad
m_2
\le
\min
\left\{
\frac{n}{p}, 
\frac{n}{2}
\right\}
-
\max
\left\{
\frac{n}{p_2},
\frac{n}{2}
\right\},
\end{align*}
then 
$
\Op(BS^{(m_1, m_2)}_{0, 0})
\subset
B(h^{p_1} \times h^{p_2} \to h^p).
$

\item
If $m_1$ and $m_2$ satisfy 
\begin{align*}
\qquad
m_1
\le
\min
\left\{
\frac{n}{p_1^{\prime}}, 
\frac{n}{2}
\right\}
-
\max
\left\{
\frac{n}{p^{\prime}},
\frac{n}{2}
\right\},
\quad
m_2
\le
\min
\left\{
\frac{n}{p_1^{\prime}}, 
\frac{n}{2}
\right\}
-
\max
\left\{
\frac{n}{p_2},
\frac{n}{2}
\right\},
\end{align*}
then 
$
\Op(BS^{(m_1, m_2), *1}_{0, 0})
\subset
B(h^{p_1} \times h^{p_2} \to h^p).
$

\item
If $m_1$ and $m_2$ satisfy 
\begin{align*}
\qquad
m_1
\le
\min
\left\{
\frac{n}{p_2^{\prime}}, 
\frac{n}{2}
\right\}
-
\max
\left\{
\frac{n}{p_1},
\frac{n}{2}
\right\},
\quad
m_2
\le
\min
\left\{
\frac{n}{p_2^{\prime}}, 
\frac{n}{2}
\right\}
-
\max
\left\{
\frac{n}{p^{\prime}},
\frac{n}{2}
\right\},
\end{align*}
then 
$
\Op(BS^{(m_1, m_2), *2}_{0, 0})
\subset
B(h^{p_1} \times h^{p_2} \to h^p).
$
\end{enumerate}
Here $h^{p}$ (resp. $h^{p_1}$, $h^{p_2}$) should be replaced by $bmo$ 
if $p=\infty$ (resp. $p_1= \infty$, $p_2 = \infty$).
\end{prop}

\section{Proof of  Theorem \ref{main-thm}}
\label{sec-mainthm}

In this section, we shall give the proof of Theorem \ref{main-thm}.
We set
\[
m_c(p_1, p_2, p) 
=
\min\{n/p, n/2\} - \max\{n/p_1, n/2\} - \max\{n/p_2, n/2\}.
\]
We restate our main result here. 
\begin{thm}\label{main-thm+}
Let $0< p_1, p_2, p \le \infty$, $1/p \le 1/p_1+1/p_2$, $s_1, s_2, s \in \R$.
\begin{enumerate}
\item
We assume that
\begin{equation}\label{m-critical*}
m 
= 
m_c(p_1, p_2, p) 
+s_1
+s_2
-s
\end{equation}
and
\begin{align}\label{m-critical?*}
s_1 < \max\left\{ \frac{n}{p_1}, \frac{n}{2} \right\},
\quad
s_2 < \max\left\{ \frac{n}{p_2}, \frac{n}{2} \right\},
\quad
s > -\max\left\{ \frac{n}{p^{\prime}}, \frac{n}{2} \right\}
\end{align}
Then 
\begin{equation}\label{main-est*}
\Op(BS^m_{0,0})
\subset
B(h^{p_1}_{s_1} \times h^{p_2}_{s_2} \to h^p_s).
\end{equation}
\item
We assume that
\begin{equation}\label{m-subcritical}
m 
< 
m_c(p_1, p_2, p)
+s_1
+s_2
-s
\end{equation}
and
\begin{equation}\label{m-subcritical?}
s_1
\le
\max
\left\{
\frac{n}{p_1},
\frac{n}{2}
\right\}
+
\kappa,
\quad
s_2
\le
\max
\left\{
\frac{n}{p_2},
\frac{n}{2}
\right\}
+\kappa,
\quad
s
\ge
-
\max
\left\{
\frac{n}{p^{\prime}},
\frac{n}{2}
\right\}
-
\kappa
\end{equation}
with $\kappa = m_c(p_1, p_2, p) +s_1+s_2-s -m > 0$.
Then the boundedness
\eqref{main-est*}
holds. 
\end{enumerate}
Here $h^{p}_{s}$ (resp. $h^{p_1}_{s_1}$, $h^{p_2}_{s_2}$) should be replaced by $bmo_{s}$ (resp. $bmo_{s_1}$, $bmo_{s_2}$) 
if $p=\infty$ (resp. $p_1=\infty$, $p_2 = \infty$).
\end{thm}

We first check that
it is sufficient to prove Theorem \ref{main-thm+} for the case that 
the symbol $\sigma \in BS^m_{0,0}$ has compact support in the variable $x$.
If the boundedness \eqref{main-est*} holds, then 
it follows from the closed graph theorem that there exists a positive integer $N$ such that
\begin{equation}\label{closedgraph*}
\|T_{\sigma}\|_{h^{p_1}_{s_1} \times h^{p_2}_{s_2} \to h^{p}_s}
\lesssim
\max_{|\alpha|, |\beta_1|, |\beta_2| \le N}
\left\|
(1+|\xi_1|+|\xi_2|)^{-m}
\pa_x^{\alpha}
\pa_{\xi_1}^{\beta_1}
\pa_{\xi_2}^{\beta_2}
\sigma(x, \xi_1, \xi_2)
\right\|_{L^{\infty}_{x,\xi_1,\xi_2}}
\end{equation}
for all $\sigma \in BS^m_{0,0}$.

For $\sigma \in BS^m_{0,0}$ and $0 < \epsilon < 1$,  
we define $\sigma_{\epsilon}(x, \xi_1, \xi_2) = \sigma(x, \xi_1, \xi_2)\vphi(\epsilon x)$,
where $\vphi \in \Sh(\R^n)$ satisfies $\supp \vphi \subset \{|x| \le 1\}$ and $\vphi(0) = 1$.
Then, it can be checked that
\[
\lim_{\epsilon \to 0} T_{\sigma_{\epsilon}}(f_1, f_2) = T_{\sigma}(f_1, f_2),
\quad
f_1, f_2 \in \Sh(\R^n)
\]
and that $\sigma_{\epsilon}$ satisfies
\[
|
\pa_{x}^{\alpha}\pa_{\xi_1}^{\beta_1}\pa_{\xi_2}^{\beta_2}
\sigma_{\epsilon}(x, \xi_1, \xi_2)
|
\lesssim
(1+|\xi_1|+|\xi_2|)^m.
\]
where the implicit constant is independent of $\epsilon$.
Combining this fact with \eqref{closedgraph*}, 
we have
$\|T_{\sigma_{\epsilon}}\|_{h^{p_1}_{s_1} \times h^{p_2}_{s_2} \to h^{p}_{s}} \lesssim 1$ 
uniformly in $\epsilon$. 
Thus Fatou's lemma yields that
\begin{align*}
\|T_{\sigma}(f_1, f_2)\|_{h^{p}_{s}}
&\le
\liminf_{\epsilon \to 0}
\|T_{\sigma_{\epsilon}}(f_1, f_2)\|_{h^{p}_{s}}
\lesssim
\|f_1\|_{h^{p_1}_{s_1}}\|f_2\|_{h^{p_2}_{s_2}},
\end{align*}
where the last $\lesssim$ does not depend on $\epsilon$. 
Thus we also obtain the desired result for $\sigma \in BS^m_{0,0}$ that do not have compact support. 
For details on such an approximate argument, see \cite{BT-1} and \cite{BMNT}.

\begin{prop}\label{tau-est}
Let $\sigma$ be in $BS^m_{0,0}$ and 
have compact support in the variable $x$, 
and let $s_1, s_2, s \in \R$.
Then, the corresponding bilinear symbol $\tau = \tau(x, \xi_1, \xi_2)$ given by \eqref{symbol***}
satisfies
\begin{equation}\label{symbol-newdecay}
|\pa_{x}^{\alpha} \pa_{\xi_1}^{\beta_1} \pa_{\xi_2}^{\beta_2} \tau(x, \xi_1, \xi_2)|
\lesssim
(1+ |\xi_1+\xi_2|)^{s}
(1+ |\xi_1| )^{-s_1}
(1+ |\xi_2| )^{-s_2}
(1+ |\xi_1| + |\xi_2|)^{m}
\end{equation}
for all multi-indices $\alpha, \beta_1, \beta_2 \in \N^n_0$.
\end{prop}

\begin{proof}
By straightforward calculations, the symbol $\tau$
can be written as 
\[
\tau(x, \xi_1, \xi_2)
=
\left(
\frac{1}{(2\pi)^{n}}
\int_{(\R^n)^2}
e^{-i z \cdot \zeta}
\la \zeta + \xi_1 + \xi_2 \ra^s
\sigma(x+z, \xi_1, \xi_2)
dz
d\zeta
\right)
\la \xi_1 \ra^{-s_1}
\la \xi_2 \ra^{-s_2}.
\]
Let $M, N \in \N_0$ be such that $2M > n$, $2N >n + |s|$. 
Since
\[
e^{-iz \cdot \zeta} 
= 
(I-\Delta_{z})^{N} e^{-iz \cdot \zeta} \la \zeta \ra^{-N} 
= 
(I-\Delta_{\zeta})^{M} e^{-iz \cdot \zeta} \la z \ra^{-M},
\]
integration by parts gives that
\begin{align*}
&\int_{(\R^n)^2}
e^{-i z \cdot \zeta}
\la \zeta + \xi_1 + \xi_2 \ra^s
\sigma(x+z, \xi_1, \xi_2)
dz
d\zeta
\\
&=
\int_{(\R^n)^2}
[(I-\Delta_{z})^{N} e^{-i z \cdot \zeta}]
\la \zeta \ra^{-2N}
\la \zeta + \xi_1 + \xi_2 \ra^s
\sigma(x+z, \xi_1, \xi_2)
dz
d\zeta
\\
&=
\int_{(\R^n)^2}
e^{-i z \cdot \zeta}
\la \zeta \ra^{-2N}
\la \zeta + \xi_1 + \xi_2 \ra^s
(I-\Delta_{z})^{N}
\sigma(x+z, \xi_1, \xi_2)
dz
d\zeta
\\
&=
\int_{(\R^n)^2}
[(I-\Delta_{\zeta})^{M}e^{-i z \cdot \zeta}]
\la z \ra^{-2M}
 \la \zeta \ra^{-2N}
\la \zeta + \xi_1 + \xi_2 \ra^s
(I-\Delta_{z})^{N}
\sigma(x+z, \xi_1, \xi_2)
dz
d\zeta
\\
&=
\int_{(\R^n)^2}
e^{-i z \cdot \zeta}
\la z \ra^{-2M}
(I-\Delta_{\zeta})^{M}
[
 \la \zeta \ra^{-2N}
\la \zeta + \xi_1 + \xi_2 \ra^s
]
[
(I-\Delta_{z})^{N}
\sigma(x+z, \xi_1, \xi_2)
]
dz
d\zeta.
\end{align*}
Here, in the second equality, we used the assumption that the symbol $\sigma$ has compact support in the variable $x$.
Hence, we obtain
\begin{align*}
&|
\pa_{x}^{\alpha} \pa_{\xi_1}^{\beta_1} \pa_{\xi_2}^{\beta_2}
\tau(x, \xi_1, \xi_2)
|
\\
&\lesssim
\left(
\int_{(\R^n)^2}
\la \zeta \ra^{-2N+|s|}
\la z \ra^{-2M}
dzd\zeta
\right)
\la \xi_1+\xi_2 \ra^{s}
\la \xi_1 \ra^{-s_1}
\la \xi_2 \ra^{-s_2}
\la (\xi_1, \xi_2) \ra^m
\\
&\lesssim
\la \xi_1+\xi_2 \ra^{s}
\la \xi_1 \ra^{-s_1}
\la \xi_2 \ra^{-s_2}
\la (\xi_1, \xi_2) \ra^m.
\end{align*}
and this implies \eqref{symbol-newdecay}.
The proof of Proposition \ref{tau-est} is complete.
\end{proof}

We define the symbol class $BS^{m, (s_1, s_2, s)}_{0,0}$ by the set of $\sigma= \sigma(x, \xi_1, \xi_2) \in C^\infty(\R^n)$ satisfying
\eqref{symbol-newdecay}.
We have the following boundedness properties of bilinear pseudo-differential operators with symbols in these classes.

\begin{prop}\label{mainprop}
Let  $0< p_1, p_2, p \le \infty$, $1/p \le 1/p_1 + 1/p_2$,
and let $s, s_1, s_2 \in \R$. 
\begin{enumerate}
\item
We assume that the number $m$ satisfies 
\eqref{m-critical*} and \eqref{m-critical?*}.
Then, the boundedness
\begin{equation}\label{bdd-newdecay}
\Op(BS^{m, (s_1, s_2, s)}_{0,0})
\subset
B(h^{p_1} \times h^{p_2} \to h^p)
\end{equation}
holds
if $p_1, p_2$ and $p$ satisfy either of the following conditions in addition:
\begin{enumerate}
\renewcommand{\labelenumii}{(1-\roman{enumii})}
\item
$1< p_1 \le \infty$ and $0< p < \infty$.
\item
$1< p_2 \le \infty$ and $0< p < \infty$.
\end{enumerate}

\item
We assume that $m \in \R$ satisfy 
\eqref{m-subcritical} and \eqref{m-subcritical?}.
Then, 
the boundedness \eqref{bdd-newdecay} holds.
\end{enumerate}
Here $h^{p}$ (resp. $h^{p_1}$, $h^{p_2}$) should be replaced by $bmo$ if $p= \infty$ (resp. $p_1 = \infty$, $p_2= \infty$).
\end{prop}

\begin{proof}

Let $\Phi \in \Sh((\R^n)^2)$ be such that
\begin{align*}
&\supp \Phi 
\subset 
\left\{
(\xi_1, \xi_2) \in (\R^n)^2 : \sqrt{|\xi_1|^2 + |\xi_2|^2} \le 2
\right\}
\\
&\Phi = 1
\quad
\text{on}
\quad
\left\{
(\xi_1, \xi_2) \in (\R^n)^2 : \sqrt{|\xi_1|^2 + |\xi_2|^2} \le 1
\right\}.
\end{align*}
We divide $\sigma$ as
\[
\sigma = \sigma\Phi + \sigma(1-\Phi).
\]
Then, since $\sigma\Phi$ is smooth and rapidly decreasing, 
it follows from the assertion (1) of Theorem \ref{est-sep} that
$T_{\sigma\Phi}$ is bounded from $h^{p_1} \times h^{p_2}$ to $h^p$.

We next consider the symbol $\sigma(1-\Phi)$.
If $(\eta_1, \eta_2)$ belongs to the unit sphere $\Sigma$ on $(\R^n)^2$, then at least two of $\eta_1 + \eta_2$, $\eta_1 $ and $\eta_2$ are not equal to 0. Hence, by the compactness of $\Sigma$,  there exists $c > 0$ such that $\Sigma$ is covered by the following three open sets 
\begin{align*}
&V_0
=
\{(\eta_1, \eta_2) \in \Sigma \,:\,
|\eta_1| > c,\ |\eta_2| > c
\},
\\
&V_1
=
\{(\eta_1, \eta_2) \in \Sigma \,:\,
|\eta_1 + \eta_2| > c,\ |\eta_1| > c
\},
\\
&V_2
=
\{(\eta_1, \eta_2) \in \Sigma \,:\,
|\eta_1 + \eta_2| > c,\ |\eta_2| > c
\}.
\end{align*}
We define $\Gamma(V_j)$, $j=0,1,2$, by
\[
\Gamma(V_j)
=
\left\{
(\xi_1, \xi_2) \in (\R^n)^2 
\, : \,
(\xi_1, \xi_2)/|(\xi_1, \xi_2)| \in V_j
\right\},
\]
and take smooth functions $\Phi_j$, $j=0,1,2$, on $\R^n \times \R^n \setminus \{(0, 0)\}$ such that
\begin{align*}
&\Phi_0(\xi_1, \xi_2) + \Phi_1(\xi_1, \xi_2) + \Phi_2(\xi_1, \xi_2) = 1
\quad
\text{for}
\quad
(\xi_1, \xi_2) \in \R^n \times \R^n \setminus \{(0, 0)\},
\\
&\supp \Phi_j \subset \Gamma(V_j), \quad j=0,1,2.
\end{align*}
Then, we decompose $\sigma(1-\Phi) = \sum_{j=0}^2 \sigma_j$ 
with $\sigma_j = \sigma(1-\Phi)\Phi_j$, $j=0, 1, 2$.
Since
\begin{align*}
&
|\xi_1+ \xi_2|
\lesssim 
|\xi_1|
\approx 
|\xi_2|
\approx 
\sqrt{|\xi_1|^2 + |\xi_2|^2}
\quad
\text{if}
\quad
(\xi_1, \xi_2) 
\in 
\supp \sigma_0(x, \cdot, \cdot),
\\
&
|\xi_2|
\lesssim 
|\xi_1|
\approx 
|\xi_1+\xi_2|
\approx 
\sqrt{|\xi_1|^2 + |\xi_2|^2}
\quad
\text{if}
\quad
(\xi_1, \xi_2) 
\in 
\supp \sigma_1(x, \cdot, \cdot),
\\
&
|\xi_1|
\lesssim 
|\xi_2|
\approx 
|\xi_1+\xi_2|
\approx 
\sqrt{|\xi_1|^2 + |\xi_2|^2}
\quad
\text{if}
\quad
(\xi_1, \xi_2) 
\in 
\supp \sigma_2(x, \cdot, \cdot),
\end{align*}
and since the symbol $\sigma$ satisfies \eqref{symbol-newdecay}, 
the symbols $\sigma_0$, $\sigma_1$ and $\sigma_2$ satisfy 
\begin{align*}
\sigma_{0} \in BS^{(s-t, m-s_1-s_2+t), *1}_{0, 0},
\quad
\sigma_{1} \in BS^{(s-s_1+m+t, -s_2-t)}_{0, 0},
\quad
\sigma_{2} \in BS^{(-s_1-t, s-s_2+m+t)}_{0, 0},
\end{align*}
for $t \ge 0$.

Now, we first prove the assertion (1).
By symmetry, it is sufficient to  prove the result under the assumption (1-i), 
that is,
we may assume that $1< p_1 \le \infty$, $0< p_2 \le \infty$, $0< p < \infty$ and $1/p \le 1/p_1+1/p_2$.

\bigskip
\textit{Estimate for $T_{\sigma_0}$}.
In order to obtain the boundedness of $T_{\sigma_0}$, 
we need the assumption $1< p_1 \le \infty$, at least by our method.

Since $s > -\max\{n/p^{\prime}, n/2\}$ and $p_1 > 1$, we can take $t_0 \ge 0$ satisfying
\begin{align}\label{st}
s
+
\max
\left\{
\frac{n}{p^{\prime}},
\frac{n}{2}
\right\}
-
\min
\left\{
\frac{n}{p_1^{\prime}},
\frac{n}{2}
\right\}
<
 t_0
<
s
+
\max
\left\{
\frac{n}{p^{\prime}},
\frac{n}{2}
\right\}.
\end{align}
Now, 
since $\sigma_0 \in BS^{(s-t_0, m-s_1-s_2+t_0), *1}_{0,0} = BS^{(s-t_0, m_c-s+t_0), *1}_{0,0}$ and
the condition \eqref{st} says that
\begin{align*}
\begin{dcases}
-
\max \left\{\frac{n}{p^\prime}, \frac{n}{2}\right\}
<
s-t_0
<
\min \left\{\frac{n}{p_1^\prime}, \frac{n}{2}\right\}
-
\max \left\{\frac{n}{p^\prime}, \frac{n}{2}\right\}
\\
- \max \left\{\frac{n}{p_2}, \frac{n}{2}\right\}
<
m_c(p_1, p_2, p)-s+t_0
<
\min \left\{\frac{n}{p_1^\prime}, \frac{n}{2}\right\}
-
\max \left\{\frac{n}{p_2}, \frac{n}{2}\right\},
\end{dcases}
\end{align*} 
the desired boundedness of $T_{\sigma_0}$ follows from the assertion (2) of Theorem \ref{est-sep}.
Here we used the simple fact that 
\begin{align*}
\max
\left\{
\frac{n}{r}, \frac{n}{2}
\right\} 
= 
n - \min\left\{\frac{n}{r^{\prime}}, \frac{n}{2}\right\}, \quad 0< r \le \infty.
\end{align*}
We remark that the assumption $1 < p_1 \le \infty$ is necessary to apply Theorem \ref{est-sep}.

\bigskip
\textit{Estimate for $T_{\sigma_1}$}.
To estimate the operator $T_{\sigma_1}$, 
it is not necessary to assume $1< p_1 \le \infty$, 
but we use the condition $0< p< \infty$ here. 

Since $s_2 < \max\{n/p_2, n/2\}$ and $p \neq \infty$, we can take $t_2 \ge 0$ satisfying
\begin{align}\label{s2t}
-\min \left\{\frac{n}{p}, \frac{n}{2}\right\}
+
\max \left\{\frac{n}{p_2}, \frac{n}{2}\right\} 
-s_2
<
 t_2
<
\max \left\{\frac{n}{p_2}, \frac{n}{2}\right\}
-s_2.
\end{align}
Now, 
since $\sigma_1 \in BS^{(s-s_1+m+t_2, -s_2-t_2)}_{0,0} = BS^{(m_c+s_2+t_2, -s_2-t_2)}_{0,0}$ and
the condition \eqref{s2t} says that
\begin{align*}
\begin{dcases}
-\max \left\{\frac{n}{p_1}, \frac{n}{2}\right\}
<
m_c(p_1, p_2, p) + s_2 + t_2 
<
\min \left\{\frac{n}{p}, \frac{n}{2}\right\}
-
\max \left\{\frac{n}{p_1}, \frac{n}{2}\right\},
\\
-\max \left\{\frac{n}{p_2}, \frac{n}{2}\right\}
<
-s_2-t_2
<
\min \left\{\frac{n}{p}, \frac{n}{2}\right\}
-
\max \left\{\frac{n}{p_2}, \frac{n}{2}\right\},
\end{dcases}
\end{align*} 
we obtain the desired boundedness of $T_{\sigma_1}$ by Theorem \ref{est-sep} (1).
Notice that we used the condition $0< p <\infty$  to apply Theorem \ref{est-sep} (1).

\bigskip
\textit{Estimate for $T_{\sigma_2}$}.
Applying the same argument as for $T_{\sigma_1}$, we obtain the desired boundedness.
We notice that the assumption $1< p_1 \le \infty$ is not necessary, however we use the assumption $0< p < \infty$ in this case again.

\bigskip

Thus, the proof of the assertion (1) is complete.

\bigskip

Next, we prove the assertion (2).

\bigskip
\textit{Estimate for $T_{\sigma_0}$}.
Since $s \ge - \max\{n/p^{\prime}, n/2 \}-\kappa$, we can take $t_0 \ge 0$ satisfying
\begin{align}\label{stsc}
s
+
\max
\left\{
\frac{n}{p^{\prime}},
\frac{n}{2}
\right\}
-
\min
\left\{
\frac{n}{p_1^{\prime}},
\frac{n}{2}
\right\}
\le
 t_0
\le
s
+
\max
\left\{
\frac{n}{p^{\prime}},
\frac{n}{2}
\right\}
+\kappa.
\end{align}
Now, 
since $\sigma_0 \in BS^{(s-t_0, m-s_1-s_2+t_0), *1}_{0,0}$ and
the condition \eqref{stsc}
can be written as 
\begin{align*}
\begin{dcases}
s-t_0
\le
\min \left\{\frac{n}{p_1^\prime}, \frac{n}{2}\right\}
-
\max \left\{\frac{n}{p^\prime}, \frac{n}{2}\right\},
\\
m-s_1-s_2+t_0
\le
\min \left\{\frac{n}{p_1^\prime}, \frac{n}{2}\right\}
-
\max \left\{\frac{n}{p_2}, \frac{n}{2}\right\},
\end{dcases}
\end{align*} 
we obtain the desired boundedness of $T_{\sigma_0}$ by Proposition \ref{sc-prop}.

\bigskip
\textit{Estimate for $T_{\sigma_1}$}.
Since $s_2 \le \max \{n/p_1, n/2\}+\kappa$, we can take $t_2 \ge 0$ satisfying
\begin{align}\label{s2tsc}
-
\min
\left\{
\frac{n}{p},
\frac{n}{2}
\right\}
+
\max
\left\{
\frac{n}{p_2},
\frac{n}{2}
\right\}
-s_2
\le
t_2
\le
\max
\left\{
\frac{n}{p_2},
\frac{n}{2}
\right\}
-s_2
+
\kappa.
\end{align}
Now, 
since $\sigma_1 \in BS^{(s-s_1+m+t_2, -s_2-t_2)}_{0,0}$ and
the condition \eqref{s2tsc}
is equivalent to 
\begin{align*}
\begin{dcases}
s-s_1+m+t_2
\le
\min \left\{\frac{n}{p}, \frac{n}{2}\right\}
-
\max \left\{\frac{n}{p_1}, \frac{n}{2}\right\},
\\
-s_2-t_2
\le
\min \left\{\frac{n}{p}, \frac{n}{2}\right\}
-
\max \left\{\frac{n}{p_2}, \frac{n}{2}\right\},
\end{dcases}
\end{align*}
we obtain the desired boundedness of $T_{\sigma_1}$ by Proposition \ref{sc-prop}.

\bigskip
\textit{Estimate for $T_{\sigma_2}$}.
We can prove the desired boundedness in the same way as for $T_{\sigma_1}$.

\bigskip
The proof of Proposition \ref{mainprop} is complete.
\end{proof}

\begin{proof}[Proof of Theorem \ref{main-thm+}]
For the assertion (1),
it follows from Propositions \ref{tau-est}
and 
\ref{mainprop}
that the boundedness
\[
\Op(BS^{m_c(p_1, p_2, p) + s_1+s_2-s}_{0,0}) \subset B(h^{p_1}_{s_1} \times h^{p_2}_{s_2} \to h^p_{s})
\]
holds with 
$(1/p_1, 1/p_2) \in [0, \infty)^2 \setminus [1, \infty)^2$,
$0< p< \infty$,
and
$s_1, s_2, s \in \R$
satisfying \eqref{m-critical?*}.
By interpolation, we obtain the desired result for all $(1/p_1, 1/p_2) \in [1, \infty)^2$ and $0< p< \infty$ (see Appendix for details).
Hence, we obtain the desired boundedness for all $(1/p_1, 1/p_2) \in [0, \infty)^2$ and $0< p < \infty$.

We finally consider the case $(1/p_1, 1/p_2) \in [0, \infty)^2$ and $p= \infty$.
Our goal here is to show the boundedness
\begin{equation}\label{bdd-infty}
\Op(BS^{m_c(p_1, p_2, \infty)+s_1+s_2-s}_{0,0})
\subset
B(h^{p_1}_{s_1} \times h^{p_2}_{s_2} \to bmo_{s})
\end{equation}
for $(1/p_1, 1/p_2) \in [0, \infty)^2$ and $s_1, s_2, s \in \R$ satisfying
\begin{equation}\label{specialcase}
s_1 
< 
\max
\left\{
\frac{n}{p_1},
\frac{n}{2}
\right\}, 
\quad 
s_2 
< 
\max
\left\{
\frac{n}{p_2},
\frac{n}{2}
\right\},
\quad
s 
> 
-n.
\end{equation}
We have already prove the boundedness 
\begin{equation*}
\Op(BS^{m_c(1, q_2, q)+t_1+t_2-t}_{0,0})
\subset
B(h^1_{t_1} \times h^{q_2}_{t_2} \to L^{q}_{t})
\end{equation*}
for $0< q_2 \le \infty$, $0 < q < \infty$ and $t_1, t_2, t \in \R$ satisfying 
$t_1 <n$,
$t_2 < \max\{n/q_2, n/2\}$ and
$t > -\max\{n/q^{\prime}, n/2\}$.
By using this assertion with $q_2=p_2$, $q=p_1^\prime$, $t_1 = -s$, $t_2 =s_2$ and $t=-s_1$, we obtain the boundedness
\begin{equation}\label{h1hp2Lp1prime}
\Op(BS^{m_c(1, p_2, p_1^{\prime})-s+s_2+s_1}_{0,0})
\subset
B(h^1_{-s} \times h^{p_2}_{s_2} \to L^{p_1^{\prime}}_{-s_1})
\end{equation}
 for $1 < p_1 \le \infty$,  $0 < p_2 \le \infty$ and $s_1, s_2, s \in \R$ satisfying \eqref{specialcase}.
Since the space $bmo_s$ is the dual space of $h^{1}_{-s}$, we obtain
\begin{equation*}
\eqref{h1hp2Lp1prime}
\iff
\Op(BS^{m_c(1, p_2, p_1^{\prime})-s+s_2+s_1}_{0,0})
\subset
B(L^{p_1}_{s_1} \times h^{p_2}_{s_2} \to bmo_s)
\end{equation*}
(for details on such a duality argument, see Section \ref{sec-necessity}).
Furthermore, we see that 
\[
m_c(1, p_2, p_1^{\prime}) -s +s_2+s_1
=
m_c(p_1, p_2, \infty) +s_1 +s_2 - s,
\]
and hence we conclude that the boundedness \eqref{bdd-infty} holds for $1 < p_1 \le \infty$,  $0 < p_2 \le \infty$ and $s_1, s_2, s \in \R$ satisfying \eqref{specialcase}.

Similarly, the boundedness  \eqref{bdd-infty} follows for $0 < p_1 \le \infty$,  $1 < p_2 \le \infty$ and $s_1, s_2, s \in \R$ satisfying \eqref{specialcase}.
Thus, the boundedness \eqref{bdd-infty}
holds with $(1/p_1, 1/p_2) \in [0, \infty)^2 \setminus [1, \infty)^2$ and $s_1, s_2, s \in \R$ satisfying \eqref{specialcase}.
Applying the interpolation argument as above, 
we obtain the boundedness \eqref{bdd-infty} 
for all $(1/p_1, 1/p_2) \in [0, \infty)^2$ and $s_1, s_2, s \in \R$ satisfying \eqref{specialcase}.
The proof of Theorem \ref{main-thm+} (1) is complete.

The assertion (2) follows from Proposition \ref{mainprop} (2).
Thus, the proof of Theorem \ref{main-thm+} is complete.
\end{proof}

%=========================================================
%=========================================================
%=========================================================
%=========================================================
%=========================================================
%=========================================================
%=========================================================
%=========================================================
%=========================================================
\section{Proof of Theorem \ref{main-thm2}}
\label{sec-necessity}
The purpose of this section is to prove Theorem \ref{main-thm2}.
We will use the Littlewood-Paley characterization of $h^p_s$ norm. 
We take the functions $\psi_k \in \Sh(\R^n)$, $k=0, 1, 2, \dots$, 
satisfying  \eqref{sumpsij} 
and satisfying
\eqref{supppsi0} and \eqref{supppsij} replaced by
\begin{align*}
&\supp \psi_{0} \subset \{|\xi| \le 2^{3/4} \},
\quad
\supp \psi_{k} \subset \{ 2^{k-3/4} \le |\xi| \le 2^{k+3/4} \},
\quad
k \ge 1, 
\end{align*}
and
\begin{align*}
&
\psi_0(\xi) = 1 
\quad
\text{on}
\quad
\{|\xi| \le 2^{1/4} \},
\quad
\psi_k(\xi) = 1 
\quad
\text{on}
\quad
\{ 2^{k-1/4} \le |\xi| \le 2^{k+1/4} \},
\quad
k \ge 1.
\end{align*}

The following proposition plays an important role to prove Theorem \ref{main-thm2}. The basic ideas of the proof below comes from \cite[Proof of Lemma 4.2]{KMT-multi-2022} and \cite[Proof of Proposition 5.1]{Shida}.

\begin{prop}\label{key-prop*}
Let $1 < p_1, p_2 < \infty$, $0< p < \infty$ and $m \in \R$. If the boundedness
\begin{equation}\label{bdd-necessity}
\Op(BS^{m}_{0,0})
\subset
B(L^{p_1}_{s_1} \times L^{p_2}_{s_2} \to h^p_s)
\end{equation}
holds, then 
\begin{align}\label{mmm}
m 
\le
\min
\left\{
\frac{n}{p},
\frac{n}{2}
\right\}
-
\max
\left\{
\frac{n}{p_1},
\frac{n}{2}
\right\}
-
\max
\left\{
\frac{n}{p_2},
\frac{n}{2}
\right\}
+s_1
+s_2
-s
\end{align}
and
\begin{align}
&m 
\le
\min
\left\{
\frac{n}{p},
\frac{n}{2}
\right\}
-
\max
\left\{
\frac{n}{p_j},
\frac{n}{2}
\right\}
+s_j
-s,
\quad
j=1,2,
\label{mppj}
\\
&m 
\le
n
-
\max
\left\{
\frac{n}{p_1},
\frac{n}{2}
\right\}
-
\max
\left\{
\frac{n}{p_2},
\frac{n}{2}
\right\}
+s_1
+s_2.
\label{mp1p2}
\end{align}
\end{prop}

\begin{rem}
If we set $\kappa = \min\{n/p, n/2\}-\max\{n/p_1, n/2\}-\max\{n/p_2, n/2\}+s_1+s_2-s -m$, then the conditions \eqref{mppj} and \eqref{mp1p2} can be written as
\begin{align*}
&s_j \le \max\{n/p_j, n/2\} + \kappa, \quad j=1,2,
\\
&s \ge -\max\{n/p^{\prime}, n/2\} -\kappa,
\end{align*}
that is, these conditions means \eqref{necessity-3} (or  \eqref{m-subcritical?}).
\end{rem}

\begin{proof}
If the boundedness \eqref{bdd-necessity} holds, then 
it follows from the closed graph theorem that there exists a positive integer $N$ such that
\begin{equation}\label{closedgraph}
\|T_{\sigma}\|_{L^{p_1}_{s_1} \times L^{p_2}_{s_2} \to h^{p}_s}
\lesssim
\max_{|\alpha|, |\beta_1|, |\beta_2| \le N}
\left\|
\la (\xi_1, \xi_2) \ra^{-m}
\pa_x^{\alpha}
\pa_{\xi_1}^{\beta_1}
\pa_{\xi_2}^{\beta_2}
\sigma(x, \xi_1, \xi_2)
\right\|_{L^{\infty}_{x,\xi_1,\xi_2}}
\end{equation}
for all $\sigma \in BS^{m}_{0,0}$ (see \cite{BMNT}).

\bigskip
\noindent
\textit{Necessity of \eqref{mmm}} :
We shall divide the proof into the following three cases;
\begin{enumerate}\renewcommand{\labelenumi}{(\roman{enumi})}
\item
$0< p \le 2$ and $(p_1, p_2) \in (1, \infty)^2$,
\item
$2 < p < \infty$ and $(p_1, p_2 ) \in (1, \infty)^2 \setminus (1,2]^2$,
\item
$2 < p < \infty$ and $(p_1, p_2 ) \in (1, 2]^2$.
\end{enumerate}

\bigskip
\noindent
\textit{Case} (i).
Let $\varphi, \widetilde{\varphi} \in \Sh(\R^n)$ be such that
\begin{align}
&\supp \varphi \subset [-1/4,1/4]^n,
\quad
|\F^{-1}\varphi(x)| \ge 1,
\quad
\text{on}
\quad
[-\pi, \pi]^n
\label{supp-vphi}
\\
&
\supp \widetilde{\varphi} \subset [-1/2,1/2]^n,
\quad
\widetilde{\varphi}(\xi) = 1 
\quad
\text{on}
\quad
\supp \varphi.
\label{supp-vphi*}
\end{align}
Let $\{c_{k_1, k_2}\}_{k_1,k_2 \in \Z^n}$ be a sequence of complex numbers satisfying $\sup_{k_1,k_2 \in \Z^n}|c_{k_1,k_2}| \le 1$.
For a sufficiently large positive integer $\ell$, 
we define the  bilinear Fourier multiplier $\sigma_{\ell}$ by
\begin{align*}
&\sigma_{\ell}(\xi_1, \xi_2)
=
\sum_{(k_1, k_2) \in D_{\ell}}
c_{k_1,k_2}
\la (k_1, k_2) \ra^{m}
\vphi(\xi_1-k_1)
\vphi(\xi_2-k_2),
\end{align*}
where
\begin{align*}
&D_{\ell}
=
\{
(k_1,k_2) \in (\Z^n)^2
\ :\ 
k_1, k_2, k_1+k_2 \in \Lambda_{\ell}
\},
\\
&\Lambda_{\ell}
=
\{
k \in \Z^n
\,
:
\,
2^{\ell-1/8}
\le
|k|
\le
2^{\ell+1/8}
\}.
\end{align*}
Since $\sup_{k_1, k_2}|c_{k_1,k_2}| \le 1$ and $|k_j| \approx |\xi_j|$ if $\xi_j \in \supp \vphi(\cdot-k_j)$, $j=1,2$, we have
\[
|\pa_{\xi_1}^{\beta_1}\pa_{\xi_2}^{\beta_2}\sigma_{\ell}(\xi_1,\xi_2)|
\lesssim
(1+|\xi_1|+|\xi_2|)^{m},
\]
that is, $\sigma_{\ell} \in BS^m_{0,0}$.
Here, it should be emphasized that the implicit constant is independent of $c_{k_1,k_2}$ and $\ell$.
Hence,  by \eqref{closedgraph}, we obtain
\begin{equation}\label{opnormest}
\|T_{\sigma_{\ell}}\|_{L^{p_1}_{s_1} \times L^{p_2}_{s_2} \to h^{p}_{s}}
\lesssim
1
\end{equation}
with the implicit constant independ of $c_{k_1,k_2}$ and $\ell$. 

Let $\phi_{\ell} \in \Sh(\R^n)$ be such that
\[
\supp \phi_{\ell} \subset \{2^{\ell-1/2} \le |\xi| \le 2^{\ell+1/2}\}
\quad
\text{and}
\quad
\phi_{\ell} = 1 
\ \ 
\text{on}
\ \ 
\{2^{\ell-1/4} \le |\xi| \le 2^{\ell+1/4}\}.
\]
For $\epsilon > 0$ and $0< a_1, a_2 <1$, we set
\[
b_j
=
\frac{n}{2}
+
(1-a_j)
\left(
\frac{n}{2}
-
\frac{n}{p_j}
\right) 
+ \epsilon,
\quad
j=1,2.
\]
We define the functions $f_{a_j, b_j, \ell}$ by
\begin{equation}\label{goodexample}
\widehat{f_{a_j, b_j, \ell}}(\xi_j) = \phi_{\ell}(\xi_j)\widehat{f_{a_j, b_j}}(\xi_j),
\quad
j=1,2,
\end{equation}
where
\begin{equation*}
f_{a_j, b_j}(x_j)
=
\sum_{\ell_j \in \Z^n \setminus \{0\}} 
|\ell_j|^{-b_j} e^{i |\ell_j|^{a_j}} 
e^{i\ell_j \cdot x_j}
\F^{-1}\widetilde{\varphi}(x_j),
\quad
j=1,2.
\end{equation*}
It was proved by Wainger \cite{Wainger} 
that  $f_{a_j, b_j} \in L^{p_j}$ (see also \cite[Lemma 4.1 and the proof of Lemma 4.2]{KMT-multi-2022}).
Thus, noticing that $|\xi_j| \approx 2^{\ell}$ if $\xi_j \in \supp \phi_{\ell}$, 
we have 
\begin{align}\label{fabequiv}
\|f_{a_j, b_j, \ell}\|_{L^{p_j}_{s_j}} 
&\approx 
\left\|
\left(
\sum_{k \ge 0}
2^{2ks_j}
|\psi_k(D)f_{a_j, b_j, \ell}|
\right)^{1/2}
\right\|_{L^{p_j}}
\lesssim
2^{\ell s_j}.
\end{align}
On the other hand, 
if $k_j \in \Lambda_\ell$, then
$\supp \vphi (\cdot - k_j) \subset \{2^{\ell-1/4} \le |\xi_j| \le 2^{\ell+1/4}\}$
since $\ell$ is sufficiently large.
Hence, we have
$
\vphi(\xi_j -k_j) \phi_{\ell}(\xi_j)
=
\vphi(\xi_j -k_j)
$
for $k_j \in \Lambda_\ell$. 
Furthermore, since $\vphi\widetilde{\vphi} = \vphi$, we have
\begin{align*}
&T_{\sigma_\ell}(f_{a_1, b_1, \ell}, f_{a_2, b_2, \ell})(x)
\\
&=
\sum_{(k_1, k_2) \in D_{\ell}}
c_{k_1,k_2}
\la (k_1, k_2) \ra^{m}
|k_1|^{-b_1} e^{i |k_1|^{a_1}} 
|k_2|^{-b_2} e^{i |k_2|^{a_2}}
e^{i(k_1+k_2) \cdot x}
\{\F^{-1}\vphi(x)\}^2.
\end{align*}
Let $\{r_k(\omega)\}_{k \in \Z^n}$ be a sequence of Rademacher functions  on $[0,1]^n$ enumerated in such a way that their index set is $\Z^n$
(for the definition of the Rademacher function, see, e.g., \cite[Appendix C]{Grafakos-Classical}).
If we choose $c_{k_1,k_2}$ as 
\[
c_{k_1,k_2}= r_{k_1+k_2}(\omega)e^{-i|k_1|^{a_1}}e^{-i|k_2|^{a_2}},
\] 
then 
\begin{align*}
T_{\sigma_\ell}(f_{a_1, b_1, \ell}, f_{a_2, b_2, \ell})(x)
&=
\sum_{(k_1, k_2) \in D_{\ell}}
r_{k_1+k_2}(\omega)
\la (k_1, k_2) \ra^{m}
|k_1|^{-b_1} 
|k_2|^{-b_2} 
e^{i(k_1+k_2) \cdot x}
\{\Phi(x)\}^2
\\
&=
\sum_{k \in \Lambda_{\ell}}
r_{k}(\omega)
e^{i k \cdot x}
\left(
\sum_{\substack{k_1, k_2 \in \Lambda_{\ell} \\ k_1+k_2=k }}
\la (k_1, k_2) \ra^{m}
|k_1|^{-b_1} 
|k_2|^{-b_2} 
\right)
\{\Phi(x)\}^2
\\
&=
\sum_{k \in \Lambda_{\ell}}
r_{k}(\omega)
e^{i k \cdot x}
d_k 
\{\Phi(x)\}^2.
\end{align*}
where $\Phi = \F^{-1}\vphi$ and
\[
d_k
=
\sum_{\substack{k_1, k_2 \in \Lambda_{\ell}  \\ k_1+k_2=k }}
\la (k_1, k_2) \ra^{m}
|k_1|^{-b_1} 
|k_2|^{-b_2}.
\]
Since $\ell$ is sufficiently large, we have
$
\supp [\vphi * \vphi](\cdot -k)
\subset
\{2^{\ell -1/4} \le |\zeta| \le 2^{\ell+1/4}\}
$ 
for 
$k \in \Lambda_{\ell}$. 
Thus we have
\[
\F[T_{\sigma_\ell}(f_{a_1, b_1, \ell}, f_{a_2, b_2, \ell})](\zeta)
=
\sum_{k \in \Lambda_{\ell}}
r_{k}(\omega)
d_k
[\vphi * \vphi](\zeta -k),
\]
and consequently, 
$\supp\F[T_{\sigma_\ell}(f_{a_1, b_1, \ell}, f_{a_2, b_2, \ell})] 
\subset 
\{2^{\ell -1/4} \le |\zeta| \le 2^{\ell+1/4}\}$.
Hence,  by the Littlewood-Paley characterization of $h^p_s$-norm, we have
\begin{align}\label{equivhps}
\begin{split}
\|T_{\sigma_\ell}(f_{a_1, b_1, \ell}, f_{a_2, b_2, \ell})\|_{h^{p}_{s}} 
&\approx
\left\|
\left(
\sum_{k =0}^{\infty}
2^{2ks}
|\psi_k(D)T_{\sigma_\ell}(f_{a_1, b_1, \ell}, f_{a_2, b_2, \ell})|^2
\right)^{1/2}
\right\|_{L^p}
\\
&=
 2^{\ell s} \|T_{\sigma_\ell}(f_{a_1, b_1, \ell}, f_{a_2, b_2, \ell})\|_{L^{p}}.
\end{split}
\end{align}
Furthermore, our assumption \eqref{supp-vphi} gives
\begin{align}\label{esttoKhintchine}
\left\|
\sum_{k \in \Lambda_{\ell}}
r_{k}(\omega)
e^{i k \cdot x}
d_k 
\{\Phi(x)\}^2
\right\|_{L^{p}_x}
&\ge
\left\|
\sum_{k \in \Lambda_{\ell}}
r_{k}(\omega)
e^{i k \cdot x}
d_k 
\right\|_{L^{p}_x[-\pi, \pi]^n}.
\end{align}
Thus, combining \eqref{equivhps}, \eqref{esttoKhintchine} and \eqref{fabequiv}, 
we obtain
\begin{align*}
\int_{[-\pi, \pi]^n}
\left|
\sum_{k \in \Lambda_{\ell}}
r_{k}(\omega)
e^{i k \cdot x}
d_k 
\right|^{p}
dx
\lesssim
2^{\ell(-s+s_1+s_2)p}.
\end{align*}
It should be emphasized that the implicit constant is independent of $\omega$ by \eqref{opnormest}. 
Hence, by integrating over $\omega \in [0,1]^n$, we obtain
\begin{equation*}
\int_{[-\pi, \pi]^n}
\int_{[0,1]^n}
\left|
\sum_{k \in \Lambda_{\ell}}
r_{k}(\omega)
e^{i k \cdot x}
d_k
\right|^{p}
d\omega
dx
\lesssim
2^{\ell(-s+s_1+s_2)p}.
\end{equation*}
It follows from Khintchine's inequality (see, e.g., \cite[Appendix C]{Grafakos-Classical})
that 
\begin{align*}
\int_{[-\pi, \pi]^n}
\int_{[0,1]^n}
\left|
\sum_{k \in \Lambda_{\ell}}
r_{k}(\omega)
e^{i k \cdot x}
d_k
\right|^{p}
d\omega
dx
&\approx
\int_{[-\pi, \pi]^n}
\left(
\sum_{k \in \Lambda_{\ell}}
|e^{ik \cdot x}d_k|^2
\right)^{p/2}
dx
\\
&\approx
\left(
\sum_{k \in \Lambda_{\ell}}
|d_k|^2
\right)^{p/2}.
\end{align*}
Since 
$|d_k| \approx 2^{\ell(m-b_1-b_2)}2^{\ell n}$ if $k \in \Lambda_{\ell}$,
we have
\[
\left(
\sum_{k \in \Lambda_{\ell}}
|d_k|^2
\right)^{1/2}
\approx
2^{\ell(m-b_1-b_2)}2^{\ell n}2^{\ell n/2}.
\]
Combining above estimates, we obtain
\[
2^{\ell(m-b_1-b_2)}2^{\ell n}2^{\ell n/2}2^{\ell(-s_1-s_2+s)}
\lesssim
1
\]
for sufficiently large $\ell$. 
Thus we obtain
\begin{align*}
m 
&\le b_1+b_2 - n- \frac{n}{2}  + s_1+s_2 -s 
\\
&=
-\frac{n}{2}
+
(1-a_1)
\left(
\frac{n}{2} - \frac{n}{p_1}
\right)
+
(1-a_2)
\left(
\frac{n}{2} - \frac{n}{p_2}
\right)
+ s_1+s_2 -s
+2\epsilon.
\end{align*}
Taking the limits as $\epsilon \to 0$ and 
\begin{align}\label{limit}
\begin{cases}
a_1 \to 1,\ a_2 \to 1 
\quad
&\text{if}\ 
 2 \le p_1, p_2 < \infty;
\\
a_1 \to 1,\ a_2 \to 0 
\quad
&\text{if}\ 
 1 < p_2 \le 2 \le p_1 < \infty;
 \\
 a_1 \to 0,\ a_2 \to 1 
\quad
&\text{if}\ 
 1 < p_1 \le 2 \le p_2 < \infty;
 \\
 a_1 \to 0,\ a_2 \to 0 
\quad
&\text{if}\ 
  1 < p_1, p_2 \le 2,
\end{cases}
\end{align}
we obtain
\[
m 
\le
\frac{n}{2}
-
\max
\left\{
\frac{n}{p_1},
\frac{n}{2}
\right\}
-
\max
\left\{
\frac{n}{p_2},
\frac{n}{2}
\right\}
+ s_1 + s_2 -s .
\]

\bigskip
\noindent
\textit{Case} (ii).
It is sufficient to consider the case $2 \le p_1 < \infty$ by symmetry.

If the bilinear operator $T_{\sigma}$ is bounded from $L^{p_1}_{s_1} \times L^{p_2}_{s_2}$ to $L^p_s$, then the operator $T_{\sigma^{*1}}$ is bounded from $L^{p^\prime}_{-s} \times L^{p_2}_{s_2}$ to $L^{p_1^\prime}_{-s_1}$ where $\sigma^{*1}$ is a bilinear symbol given by
\[
\int_{\R^n}
T_{\sigma}(f_1, f_2)(x) g(x)
\, dx
=
\int_{\R^n}
T_{\sigma^{*1}}(g, f_2)(x) f_1(x)
\, dx.
\]
Since the class $BS^m_{0,0}$ is invariant under the transposition $\sigma \mapsto \sigma^{*1}$ (see \cite[Theorem 1.2]{BMNT}), 
the boundedness \eqref{bdd-necessity} in this case implies
\[
\Op(BS^{m}_{0,0})
\subset
B(L^{p^{\prime}}_{-s} \times L^{p_2}_{s_2} \to L^{p_1^\prime}_{-s_1}).
\]
Since $1< p_1^{\prime} \le 2$ and $1 <p^\prime \le 2$, it follows from the \textit{Case} (i) that
\begin{align*}
m 
&\le
\frac{n}{2}
-
\frac{n}{p^\prime}
-
\max
\left\{
\frac{n}{p_2},
\frac{n}{2}
\right\}
+ (-s) +s_2 -(-s_1) 
\\
&=
\frac{n}{p}
-
\frac{n}{2}
-
\max
\left\{
\frac{n}{p_2},
\frac{n}{2}
\right\}
+ s_1 +s_2 -s, 
\end{align*}
which implies the desired conclusion.

\bigskip
\noindent
\textit{Case} (iii). 
Let $\Psi \in \Sh((\R^n)^2)$ be such that
\begin{align*}
&\supp \Psi \subset 
\{
2^{-1/4} \le |\xi_1| \le 2^{1/4},
\ \ 
2^{-4-1/4} \le |\xi_2| \le 2^{-4+1/4}
\},
\\
&\Psi(\xi_1, \xi_2) = 1 
\quad
\text{on}
\quad
\{
2^{-1/8} \le |\xi_1| \le 2^{1/8},
\ \ 
2^{-4-1/8} \le |\xi_2| \le 2^{-4+1/8}
\}.
\end{align*}
We take $\widetilde{\psi}_j \in \Sh(\R^n)$, $j=1,2$, satisfying
\[
\supp \widetilde{\psi}_1 \subset \{ 2^{-1/8}  \le |\xi_1| \le 2^{1/8}\}
\quad
\text{and}
\quad
\supp \widetilde{\psi}_2 \subset \{ 2^{-4-1/8}  \le |\xi_2| \le 2^{-4+1/8}\}, 
\]
and set
\begin{align*}
&\sigma(\xi_1, \xi_2)
=
\sum_{k=0}^\infty
2^{km}
\Psi(2^{-k}\xi_1, 2^{-k}\xi_2),
\\
&
\widehat{f_{j, \ell}}(\xi_j)
=
2^{\ell n(1/p_j -1)}
\widetilde{\psi}_j(2^{-\ell} \xi_j), 
\quad
j=1,2.
\end{align*}
By the support conditions of $\widetilde{\psi}_1$ and $\widetilde{\psi}_2$, 
we have
$\|f_{j, \ell}\|_{L^{p_j}_{s_j}} \approx 2^{\ell s_j}$.
It follows from the support conditions of $\Psi$ and $\widetilde{\psi}$ that
\begin{align*}
T_{\sigma}(f_{1, \ell}, f_{2, \ell})(x)
=
2^{\ell (m+n/p_1+n/p_2)}
[\F^{-1}\widetilde{\psi}_1](2^{\ell}x)
[\F^{-1}\widetilde{\psi}_2](2^{\ell}x).
\end{align*}
Since 
\[
\supp \F[T_{\sigma}(f_{1, \ell}, f_{2, \ell})]
\subset
\supp [\widetilde{\psi}_1 * \widetilde{\psi}_2](2^{-\ell} \cdot) 
\subset 
\{ 2^{\ell-1/4} \le |\xi| \le 2^{\ell+1/4}\},
\]
we obtain
\begin{align*}
\|T_{\sigma}(f_{1, \ell}, f_{2,\ell})\|_{L^{p}_s}
&\approx
\left\|
\left(
\sum_{k=0}^\infty
2^{2ks}
|\psi_k(D)T_{\sigma}(f_{1, \ell}, f_{2, \ell})|^2
\right)^{1/2}
\right\|_{L^p}
\\
&=
2^{\ell (m + n/p_1 +n/p_2-n/p +s)}
\left\|(\F^{-1}\widetilde{\psi}_1)(\F^{-1}\widetilde{\psi}_2)\right\|_{L^{p}}.
\end{align*}
Thus, our assumption gives that
\[
2^{\ell (m + n/p_1 +n/p_2-n/p +s)}
\approx
\|T_{\sigma}(f_1, f_2)\|_{L^{p}_s}
\lesssim
\|f_1\|_{L^{p_1}_{s_1}}
\|f_2\|_{L^{p_2}_{s_2}}
\approx
2^{\ell(s_1+s_2)},
\]
and consequently we conclude that
\[
m \le
\frac{n}{p}
-
\frac{n}{p_1}
-
\frac{n}{p_2}
+s_1+s_2-s.
\]

\bigskip
The proof of the necessity of \eqref{mmm} is complete.

\bigskip
\noindent
\textit{Necessity of \eqref{mp1p2}} :
For a large $\ell \in \N$, we set 
\begin{align*}
&\sigma_\ell(\xi_1, \xi_2)
=
\sum_{\mu \in \Lambda_\ell}
c_{\mu} \la \mu \ra^{m}
\vphi(\xi_1- \mu)
\vphi(\xi_2+ \mu),
\end{align*}
where $\{c_{\mu}\}_{\mu \in \Z^n}$ is a sequence of complex numbers satisfying $\sup_{\mu \in \Z^n} |c_\mu| \le 1$. 
Then, we can check that $\sigma_{\ell} \in BS^m_{0,0}$.
We use the functions $f_{a_j, b_j, \ell}$, $j=1,2$, given in \eqref{goodexample}.
Since $\supp \vphi(\cdot \pm \mu) \subset \{2^{\ell-1/4} \le |\xi| \le 2^{\ell +1/4}\}$ if $\mu \in \Lambda_\ell$, then we have
\[
\varphi(\xi_1 - \mu)\varphi(\xi_2 + \mu)\phi_{\ell}(\xi_1)\phi_{\ell}(\xi_2)= \varphi(\xi_1 - \mu)\varphi(\xi_2 + \mu)
\] 
Hence, we have by the support properties of $\vphi$ and $\widetilde{\vphi}$
\begin{align*}
&T_{\sigma_\ell}(f_{a_1, b_1, \ell}, f_{a_2, b_2, \ell})(x)
\\
&=
\sum_{\mu \in \Lambda_\ell}
c_{\mu} \la \mu \ra^{m}
|\mu|^{-b_1-b_2}e^{i(|\mu|^{a_1}+|\mu|^{a_2})}
\int_{\R^{2n}}
e^{i x \cdot (\xi_1+\xi_2)}
\vphi(\xi_1- \mu)
\vphi(\xi_2+ \mu)
\,
d\xi_1
d\xi_2
\\
&=
\sum_{\mu \in \Lambda_\ell}
c_{\mu} \la \mu \ra^{m}
|\mu|^{-b_1-b_2}e^{i(|\mu|^{a_1}+|\mu|^{a_2})}
\{\Phi(x)\}^2,
\end{align*}
where $\Phi = \F^{-1} \varphi$.
Thus, if we choose $c_\mu = e^{-i(|\mu|^{a_1}+|\mu|^{a_2})}$, then
\begin{align*}
T_{\sigma_\ell}(f_{a_1, b_1, \ell}, f_{a_2, b_2, \ell})(x)
=
\sum_{\mu \in \Lambda_\ell}
\la \mu \ra^{m}
|\mu|^{-b_1-b_2}
\{\Phi(x)\}^2.
\end{align*}
Thus, we have
\[
\|T_{\sigma_\ell}(f_{a_1, b_1, \ell}, f_{a_2, b_2, \ell})\|_{h^p_{s}}
= 
c
\left(
\sum_{\mu \in \Lambda_\ell}
\la \mu \ra^{m}
|\mu|^{-b_1-b_2}
\right).
\] 
with $c = \|\Phi^2\|_{h^{p}_s}$.
Thus, since $|\mu| \approx 2^{\ell}$ if $\mu \in \Lambda_{\ell}$, and since $|\Lambda_{\ell}| \approx 2^{\ell n}$, we obtain
\begin{align*}
2^{\ell(m-b_1-b_2+n)}
\lesssim
2^{\ell(s_1+s_2)}.
\end{align*}
Since $\ell$ is arbitrarily large, we obtain
\begin{align*}
m
&\le 
b_1+b_2-n+s_1+s_2
\\
&=
(1-a_1)
\left(
\frac{n}{2}
-
\frac{n}{p_1}
\right)
+
(1-a_2)
\left(
\frac{n}{2}
-
\frac{n}{p_2}
\right)
+
s_1+s_2
+2\epsilon.
\end{align*}
Taking the limits as in \eqref{limit} and $\epsilon \to 0$, 
we conclude that
\[
m
\le
n
-
\max
\left\{
\frac{n}{p_1},
\frac{n}{2}
\right\}
-
\max
\left\{
\frac{n}{p_2},
\frac{n}{2}
\right\}
+s_1+s_2.
\]

\bigskip
\noindent
\textit{Necessity of \eqref{mppj}} :
It suffices to prove the necessity of the condition \eqref{mppj} with $j=2$, by symmetry.

For the case  $1 < p < \infty$, we again use the duality argument, 
which is used in the proof of the necessity of \eqref{mmm}.
By duality, we see that 
the boundedness \eqref{bdd-necessity} implies
\[
\Op(BS^{m}_{0,0})
\subset
B
(L^{p^\prime}_{-s} \times L^{p_2}_{s_2} \to L^{p_1^\prime}_{-s_1}).
\]
Since $1< p^{\prime} < \infty$,  it follows from the necessity of \eqref{mp1p2} that  
\[
m
\le 
n
-
\max
\left\{
\frac{n}{p^\prime},
\frac{n}{2}
\right\}
-
\max
\left\{
\frac{n}{p_2},
\frac{n}{2}
\right\}
+
(-s)
+
s_2,
\]
and this means
\[
m
\le 
\min
\left\{
\frac{n}{p},
\frac{n}{2}
\right\}
-
\max
\left\{
\frac{n}{p_2},
\frac{n}{2}
\right\}
+
s_2
-s.
\]

Next, we consider the case  $0 < p \le 1$. 
Notice that $\min\{n/p, n/2\} = n/2$.
We set 
\begin{align*}
&\sigma_{\ell}(\xi_1, \xi_2)
=
\vphi(\xi_1)
\sum_{\mu \in \Lambda_\ell}
c_\mu \la \mu \ra^{m}
\vphi(\xi_2-\mu), 
\end{align*}
where $\{c_{\mu}\}_{\mu \in \Z^n}$ is a sequence of complex numbers satisfying $\sup_{\mu \in \Z^n} |c_\mu| \le 1$. 
We use the functions  
\begin{align*}
&
\widehat{f}_1(\xi_1) = \widetilde{\vphi}(\xi_1),
\quad
\widehat{f_{a_2, b_2, \ell}}(\xi_2) = \phi_\ell(\xi_2) \widehat{f_{a_2, b_2}}(\xi_2)
\end{align*}
($\widetilde{\vphi}$ and $f_{a_2, b_2, \ell}$  are the same as in \eqref{supp-vphi*}  and \eqref{goodexample}, respectively).
Since $\sup_{\mu \in \Z^n}|c_\mu| \le 1$, we have  
\[
|\pa_{\xi_1}^{\beta_1} \pa_{\xi_2}^{\beta_2} \sigma_{\ell}(\xi_1, \xi_2)|
\lesssim
(1+|\xi_1| + |\xi_2|)^m
\] 
uniformly in $c_{\mu}$ and $\ell$.
Furthermore, by \eqref{closedgraph},  we have
\begin{equation}\label{opnormuniform}
\|T_{\sigma_{\ell}}\|_{L^{p_1}_{s_1} \times L^{p_2}_{s_2} \to h^p_s}
\lesssim
1
\end{equation} 
with the implicit constant independent of $c_{\mu}$ and $\ell$.

Now, since $\vphi\widetilde{\vphi} = \vphi$, and since $\vphi(\xi_2-\mu)\phi_{\ell}(\xi_2) = \vphi(\xi_2-\mu)$ if $\mu \in \Lambda_{\ell}$, we have
\[
T_{\sigma_\ell}(f_1, f_{a_2, b_2, \ell})(x)
=
\{\Phi(x)\}^2
\sum_{\mu \in \Lambda_\ell}
c_\mu 
\la \mu \ra^{m}
|\mu|^{-b_2}e^{i|\mu|^{a_2}}
e^{i \mu \cdot x}.
\]
We choose $\{c_{\mu}\}$ as $c_\mu = r_\mu(\omega) e^{-i |\mu|^{a_2}}$,
where $\{r_{\mu}(\omega)\}$ is a sequence of Rademacher functions on 
$[0, 1]^n$ enumerated in such a way that their index set is $\Z^n$.
Then, we have
\begin{align*}
T_{\sigma_\ell}(f_1, f_{a_2, b_2, \ell})(x)
=
\sum_{\mu \in \Lambda_\ell}
r_\mu(\omega) \la \mu \ra^{m}
|\mu|^{-b_2}
e^{i \mu \cdot x}
\{\Phi(x)\}^2.
\end{align*}
Since
\[
\supp \F[T_{\sigma_\ell}(f_1, f_{a_2, b_2, \ell})]
\subset
\bigcup_{\mu \in \Lambda_{\ell}}\supp[\vphi * \vphi](\cdot-\mu)
\subset
\{
2^{\ell-1/4}
\le
|\zeta|
\le
2^{\ell+1/4}
\},
\] 
we have 
\begin{equation}\label{saiaku}
\|T_{\sigma_\ell}(f_1, f_{a_2, b_2, \ell})\|_{h^{p}_{s}} 
\approx 
2^{\ell s}\|T_{\sigma_\ell}(f_1, f_{a_2, b_2, \ell})\|_{L^p}.
\end{equation}
Furthermore, by the same argument as in the proof of the necessity \eqref{mmm}, 
we have 
\begin{align*}
\int_{[0,1]^n}
\|T_{\sigma_\ell}(f_1, f_{a_2, b_2, \ell})\|_{L^p}^{p}
d\omega
&=
\int_{[0,1]^n}
\int_{\R^n}
\left|
\sum_{\mu \in \Lambda_\ell}
r_\mu(\omega) \la \mu \ra^{m}
|\mu|^{-b_2}
e^{i \mu \cdot x}
\{\Phi(x)\}^2
\right|^p
d\omega dx
\\
&\ge
\int_{[-\pi, \pi]^n}
\int_{[0,1]^n}
\left|
\sum_{\mu \in \Lambda_\ell}
r_\mu(\omega) \la \mu \ra^{m}
|\mu|^{-b_2}
e^{i \mu \cdot x}
\right|^p
d\omega dx
\\
&\approx
\int_{[-\pi, \pi]^n}
\left(
\sum_{\mu \in \Lambda_\ell}
\la \mu \ra^{2m}
|\mu|^{-2b_2}
\right)^{p/2}
dx
\\
&\approx
2^{(m-b_2+n/2)p}.
\end{align*}
where we used the assumption \eqref{supp-vphi} in the second inequality
and 
 Khintchine's inequality in the third inequality.
Thus, combining the above estimate with \eqref{opnormuniform}, \eqref{saiaku} and \eqref{fabequiv},
we obtain 
\[
2^{\ell(s+m-b_2+n/2) p}
\lesssim
\int_{[0,1]^n}
\|T_{\sigma_\ell}(f_1, f_{a_2, b_2, \ell})\|_{h^p_s}^{p}
d\omega
\lesssim
\|f_1\|_{L^{p_1}_{s_1}}^{p}
\|f_{a_2, b_2, \ell}\|_{L^{p_2}_{s_2}}^{p}
\lesssim
2^{\ell s_2 p},
\]
that is, 
$2^{\ell(s+m-b_2+n/2)} \lesssim 2^{\ell s_2}$.
Since $\ell$ is arbitrarily large, we obtain
\[
m
\le 
b_2-\frac{n}{2} + s_2 -s
=
(1-a_2)
\left(
\frac{n}{2}-\frac{n}{p_2}
\right) 
+s_2-s
+ \epsilon.
\]
If we take limits as $\epsilon \to 0$ and
\begin{align*}
\begin{cases}
a_2 \to 1 &\text{if} \quad 2 \le p_2 < \infty,
\\ 
a_2 \to 0 &\text{if} \quad 1 < p_2 \le 2, 
\end{cases}
\end{align*}
then
we obtain
\[
m
\le
\frac{n}{2}
-
\max
\left\{
\frac{n}{p_2},
\frac{n}{2}	
\right\}
+s_2
-s.
\]

\bigskip
The proof of Proposition \ref{key-prop*} is complete.
\end{proof}

\begin{proof}[Proof of Theorem \ref{main-thm2}]
Let $0< p, p_1, p_2 \le \infty$, $1/p \le 1/p_1+1/p_2$ and $s_1, s_2, s \in \R$.

We first consider the necessity of the condition \eqref{necessity-1} (or \eqref{mmm}).
We set
\[
m
=
\min
\left\{
\frac{n}{p},
\frac{n}{2}
\right\}
-
\max
\left\{
\frac{n}{p_1},
\frac{n}{2}
\right\}
-
\max
\left\{
\frac{n}{p_2},
\frac{n}{2}
\right\}
+s_1
+s_2
-s
+\epsilon.
\]
for some $\epsilon >0$.
Toward a contradiction, we suppose that the boundedness \eqref{main-est} with this $m$.
It is proved in \cite[Theorem 1.3 and Example 1.4]{KMT-bi-L2L2} that the boundedness
\[
\Op(BS^{-n/2}_{0,0})
\subset
B(L^{2} \times L^{2} \to L^{2})
\]
holds. Then, by complex interpolation, 
\[
\Op(BS^{\widetilde{m}}_{0,0})
\subset
B(L^{\widetilde{p}_1}_{\widetilde{s}_1} \times L^{\widetilde{p}_2}_{\widetilde{s}_2} \to h^{\widetilde{p}}_{\widetilde{s}}),
\]
where
\begin{align*}
&\frac{1}{\widetilde{p}_j}
=
\frac{1-\theta}{2} + \frac{\theta}{p_j},
\quad
\widetilde{s}_j = \theta s_j,
\quad
j=1,2,
\\
&
\frac{1}{\widetilde{p}}
=
\frac{1-\theta}{2} + \frac{\theta}{p},
\quad
\widetilde{s} = \theta s,
\end{align*}
and
\begin{align*}
\widetilde{m}
&=
(1-\theta)
\left(
-
\frac{n}{2}
\right)
+
\theta
m
\\
&=
\min
\left\{
\frac{n}{\widetilde{p}},
\frac{n}{2}
\right\}
-
\max
\left\{
\frac{n}{\widetilde{p}_1},
\frac{n}{2}
\right\}
-
\max
\left\{
\frac{n}{\widetilde{p}_2},
\frac{n}{2}
\right\}
+\widetilde{s}_1
+\widetilde{s}_2
-\widetilde{s}
+\theta \epsilon.
\end{align*}
and $0< \theta < 1$. 
If we take $\theta$ sufficiently close to $0$, then $1< \widetilde{p}_1, \widetilde{p}_2 < \infty$.
However, by Proposition \ref{key-prop*}, it is impossible to have the boundedness with these $\widetilde{p}_1, \widetilde{p}_2, \widetilde{p}$ and $\widetilde{m}$. 
Thus, the proof of the necessity of the condition \eqref{necessity-1} is complete.

We next prove the necessity of the condition \eqref{necessity-3}. 
We suppose that the boundedness \eqref{main-est} holds for 
\[
m 
=
\min
\left\{
\frac{n}{p},
\frac{n}{2}
\right\}
-
\max
\left\{
\frac{n}{p_1},
\frac{n}{2}
\right\}
+s_1-s
+\epsilon, 
\quad
\epsilon > 0.
\]
It is proved in Theorem \ref{main-thm} that the boundedness 
\begin{equation}\label{L2tL2tL2t}
\Op(BS^{-n/2+t}_{0, 0})
\subset
B(L^{2}_{t} \times L^{2}_{t} \to L^{2}_{t}).
\end{equation}
holds for  any $-n/2 < t < n/2$. 
Then, by complex interpolation, we have
\[
\Op(BS^{\widetilde{m}}_{0, 0})
\subset
B(h^{\widetilde{p}_1}_{\widetilde{s}_1} \times h^{\widetilde{p}_2}_{\widetilde{s}_2} \to h^{\widetilde{p}}_{\widetilde{s}})
\]
where $0< \theta < 1$, 
\begin{align*}
&\frac{1}{\widetilde{p}_j}
=
\frac{1-\theta}{2}
+
\frac{\theta}{p_j},
\quad
\widetilde{s_j} 
=
(1-\theta)t
+
\theta s_j ,
\quad
j=1, 2,
\\
&\frac{1}{\widetilde{p}}
=
\frac{1-\theta}{2}
+
\frac{\theta}{p},
\quad
\widetilde{s} 
=
(1-\theta)t
+
\theta s,
\end{align*}
and
\begin{align*}
\widetilde{m}
&=
(1-\theta)
\left(
-\frac{n}{2}
+t
\right)
+
\theta
m
\\
&=
\min
\left\{
\frac{n}{\widetilde{p}},
\frac{n}{2}
\right\}
-
\max
\left\{
\frac{n}{\widetilde{p}_1},
\frac{n}{2}
\right\}
+\widetilde{s}_1- \widetilde{s}
+
(1-\theta)
\left(
t
-
\frac{n}{2}
\right)
+
\theta\epsilon.
\end{align*}
Notice that $1 < \widetilde{p}_1, \widetilde{p}_2 < \infty$ if $\theta$ is sufficiently close to $0$.
If we take the number $t$ such that
\[
\frac{n}{2}
-
\frac{\theta \epsilon}{1-\theta}
<
t
<
\frac{n}{2},
\]
then $\widetilde{m} > \min\{n/\widetilde{p}, n/2\} -\max\{n/\widetilde{p}_1, n/2\} +\widetilde{s}_1-\widetilde{s}$,
but this contradicts  Proposition \ref{key-prop*} since $1< \widetilde{p}_1, \widetilde{p}_2 < \infty$. Thus we obtain the necessity of the condition \eqref{mppj} with $j=1$.

The necessity of the condition \eqref{mppj} with $j=2$ follows from the same argument as above.

Finally, we consider the necessity of the condition \eqref{mp1p2}.
Set
\[
m 
=
n
-
\max
\left\{
\frac{n}{p_1},
\frac{n}{2}
\right\}
-
\max
\left\{
\frac{n}{p_2},
\frac{n}{2}
\right\}
+s_1+s_2
+\epsilon, 
\quad
\epsilon > 0.
\]
We suppose that the boundedness \eqref{main-est} holds for this $m$.
Then, interpolating \eqref{main-est} with \eqref{L2tL2tL2t}, we obtain
\[
\Op(BS^{\widetilde{m}}_{0, 0})
\subset
B(h^{\widetilde{p}_1}_{\widetilde{s}_1} \times h^{\widetilde{p}_2}_{\widetilde{s}_2} \to h^{\widetilde{p}}_{\widetilde{s}})
\]
with $0< \theta < 1$ and
\begin{align*}
&\frac{1}{\widetilde{p}_j}
=
\frac{1-\theta}{2}
+
\frac{\theta}{p_j},
\quad
\widetilde{s_j} 
=
(1-\theta)t
+
\theta s_j,
\quad
j=1, 2,
\\
&\frac{1}{\widetilde{p}}
=
\frac{1-\theta}{2}
+
\frac{\theta}{p},
\quad
\widetilde{s} 
=
(1-\theta)t
+
\theta s
\end{align*}
and
\begin{align*}
\widetilde{m}
&=
(1-\theta)
\left(
-\frac{n}{2}
+t
\right)
+
\theta
m
\\
&=
n
-
\max
\left\{
\frac{n}{\widetilde{p}_1},
\frac{n}{2}
\right\}
-
\max
\left\{
\frac{n}{\widetilde{p}_2},
\frac{n}{2}
\right\}
+\widetilde{s}_1
+
\widetilde{s}_2
-
(1-\theta)
\left(
t
+
\frac{n}{2}
\right)
+
\theta\epsilon
\end{align*}
Taking the number $t$ satisfying  
\[
-\frac{n}{2}
<
t
<
-\frac{n}{2}
+
\frac{\theta \epsilon}{1-\theta},
\]
we have
$\widetilde{m} > n - \max\{n/\widetilde{p}_1, n/2\} -\max\{n/\widetilde{p}_2, n/2\} +\widetilde{s}_1+\widetilde{s}_2$.
However, this is a contradiction by the same reasons just above. 
Hence we obtain the necessity  of the condition \eqref{necessity-3}.

The proof of Theorem \ref{main-thm2} is complete.
\end{proof}

\section{Appendix }
In this appendix, we complete the proof of Theorem \ref{main-thm+}.

\begin{proof}[Proof of Theorem \ref{main-thm+} with $(1/p_1, 1/p_2) \in [1, \infty)^2$]
Let $0 < p_1, p_2 \le 1$ and $0< p \le \infty$ be such that  $1/p \le 1/p_1 + 1/p_2$, 
and
let $s_1, s_2, s \in \R$ be such that 
\begin{align}\label{apdassum}
s_1 <  \frac{n}{p_1}, 
\quad
s_2 <  \frac{n}{p_2}, 
\quad
s > - \max \left\{ \frac{n}{p^{\prime}}, \frac{n}{2} \right\}.
\end{align}
We prove 
\[
\Op(BS^{m_c(p_1, p_2, p)+s_1+s_2-s}_{0,0})
\subset
B(h^{p_1}_{s_1} \times h^{p_2}_{s_2} \to h^p_s),
\]
where $m_c(p_1, p_2, p)$ is the same defined in the Section \ref{sec-mainthm}.

We set
\begin{align*}
&
\left(
\frac{1}{q_1}, 
\frac{1}{q_2}
\right)
=
\left(
\frac{1}{p_1} + \frac{1}{p_2} - \frac{1}{2},
\ 
\frac{1}{2}
\right),
\quad
(t_1, t_2)
=
\left(
s_1 +\frac{n}{p_2} - \frac{n}{2},
\ 
s_2 + \frac{n}{2} - \frac{n}{p_2}
\right),
\\
&
\left(
\frac{1}{r_1}, 
\frac{1}{r_2}
\right)
=
\left(
\frac{1}{2},
\ 
\frac{1}{p_1} + \frac{1}{p_2} - \frac{1}{2}
\right),
\quad
(u_1, u_2)
=
\left(
s_1 +\frac{n}{2} - \frac{n}{p_1},
\ 
s_2 +\frac{n}{p_1} - \frac{n}{2}
\right).
\end{align*}
Then since $1/q_1 + 1/q_2 = 1/r_1 + 1/r_2 = 1/p_1+ 1/p_2$, we have
$1/p \le 1/q_1+1/q_2$ and $ 1/p \le 1/r_1+1/r_2$. 
We see that $(1/q_1, 1/q_2)$ and $(1/r_1, 1/r_2)$ are in
$[0, \infty)^2 \setminus [1, \infty)^2$.
Furthermore our assumption \eqref{apdassum} implies that 
\begin{align*}
&
\quad
t_1 < \frac{n}{q_1},
\quad
t_2 < \frac{n}{q_2}
\quad
\text{and}
\quad
u_1 < \frac{n}{r_1},
\quad
u_2 < \frac{n}{r_2}.
\end{align*}
Hence,  as shown in the proof of Theorem \ref{main-thm+}, 
we have the following boundedness.
\begin{align}
&
\Op(BS^{m_c(q_1, q_2, p)+t_1+t_2-s}_{0,0})
\subset
B(h^{q_1}_{t_1} \times h^{q_2}_{t_2} \to h^{p}_{s});
\label{interpbdd1}
\\
&
\Op(BS^{m_c(r_1, r_2, p)+u_1+u_2-s}_{0,0})
\subset
B(h^{r_1}_{u_1} \times h^{r_2}_{u_2} \to h^{p}_{s}).
\label{interpbdd2}
\end{align}
If we set  
$\theta = (1/p_2 - 1/2)/(1/p_1 + 1/p_2 - 1)$, 
then $0< \theta < 1$. 
Thus, interpolating \eqref{interpbdd1} and \eqref{interpbdd2} with this $\theta$, we obtain the desired boundedness since
\begin{align*}
&
\frac{1-\theta}{q_j}
+
\frac{\theta}{r_j}
=
\frac{1}{p_j},
\quad
(1-\theta) t_j + \theta u_j
=
s_j,
\quad
j=1,2,
\end{align*}
and
\begin{align*}
&(1-\theta)(m_c(q_1, q_2, p) + t_1+t_2-s)
+
\theta (m_c(r_1, r_2, p) + u_1+u_2-s)
\\
&=
m_c(p_1, p_2, p) + s_1+s_2-s.
\end{align*}
\end{proof}

\section*{Acknowledgement}
The author sincerely expresses his thanks to Professor Naohito Tomita for valuable suggestion and fruitful comments.
This work was supported by Grant-in-Aid for JSPS Fellows 
Grant Number 22J10001.

\end{document}